\documentclass[11pt]{article}
\usepackage{mathtools}
\usepackage[T1]{fontenc}
\usepackage{amsfonts}
\usepackage{amsmath}
\usepackage{amssymb}
\usepackage{amsthm}
\usepackage{bbm}
\usepackage{bm}
\usepackage{mathrsfs}
\usepackage{color}
\usepackage{pdfsync}
\usepackage{enumitem}
\usepackage[colorlinks=true, linkcolor=blue, citecolor=blue, urlcolor=blue]{hyperref}

\usepackage{tikz}
\usepackage{subcaption}

\DeclareMathOperator*{\argmax}{argmax}

\DeclareMathOperator*{\diam}{diam}
\DeclareMathOperator{\proj}{proj}

\newcommand{\RR}{\mathbb{R}}
\newcommand{\R}{\RR}

\newcommand{\eps}{\varepsilon}
\newcommand{\cC}{\mathcal{C}}
\newcommand{\cI}{\mathcal{I}}
\newcommand{\cT}{\mathcal{T}}

\newcommand{\cL}{\mathcal{L}}
\newcommand{\cR}{\mathcal{R}}
\newcommand{\cS}{\mathcal{S}}

\newcommand{\Id}{\mathbbm{I}}
\newcommand{\Var}{{\rm Var}}
\newcommand{\op}{{\rm op}}
\newcommand{\Hil}{\mathcal{H}}
\newcommand{\1}{\mathbf{1}}
\newcommand{\DUAL}{{\rm QOT}^{\star}} 
\newcommand{\QOT}{{\rm QOT}}

\newcommand{\proofreadhere}{}

\usepackage[top=30mm,bottom=30mm,left=31mm,right=31mm]{geometry}
\newcommand{\mykill}[1]{}
\usepackage[capitalize, noabbrev]{cleveref}
\crefname{equation}{}{} %

\theoremstyle{plain}
\newtheorem{theorem}{Theorem}[section]
\newtheorem{proposition}[theorem]{Proposition}
\newtheorem{lemma}[theorem]{Lemma}
\newtheorem{corollary}[theorem]{Corollary}
\theoremstyle{definition}

\newtheorem{remark}[theorem]{Remark}

\newtheorem{assumption}[theorem]{Assumption}
\crefname{assumption}{Assumption}{Assumptions}
\Crefname{assumption}{Assumption}{Assumptions}
\theoremstyle{remark}
\AddToHook{env/theorem/begin}{\crefalias{section}{theorem}}
\AddToHook{env/proposition/begin}{\crefalias{theorem}{proposition}}
\AddToHook{env/lemma/begin}{\crefalias{theorem}{lemma}}
\AddToHook{env/corollary/begin}{\crefalias{theorem}{corollary}}
\AddToHook{env/definition/begin}{\crefalias{theorem}{definition}}
\AddToHook{env/remark/begin}{\crefalias{theorem}{remark}}
\AddToHook{env/example/begin}{\crefalias{theorem}{example}}
\AddToHook{env/assumption/begin}{\crefalias{theorem}{assumption}}
\crefname{theorem}{Theorem}{Theorems}
\crefname{proposition}{Proposition}{Propositions}
\crefname{lemma}{Lemma}{Lemmas}
\crefname{corollary}{Corollary}{Corollaries}
\crefname{definition}{Definition}{Definitions}
\crefname{remark}{Remark}{Remarks}
\crefname{example}{Example}{Examples}
\crefname{assumption}{Assumption}{Assumptions}

\renewenvironment{thebibliography}[1]{%
\begin{oldthebibliography}{#1}%
\setlength{\baselineskip}{.9em}
\linespread{1}
\small
\setlength{\parskip}{.40ex}%
\setlength{\itemsep}{.30em}%
}%
{%
\end{oldthebibliography}%
}

\begin{document}

\title{\vspace{-2.5em}Polyak--\L{}ojasiewicz Inequality\\for Quadratically Regularized Optimal Transport}
\date{\today}
\author{  
  Alberto Gonz{\'a}lez-Sanz%
  \thanks{Department of Statistics, Columbia University, ag4855@columbia.edu.} \and  Marcel Nutz%
  \thanks{Departments of Mathematics and Statistics, Columbia University, mnutz@columbia.edu. Research supported by NSF Grants DMS-2106056, DMS-2407074.}   \and Andrés Riveros Valdevenito\thanks{Department of Statistics, Columbia University, ar4151@columbia.edu.}
  }
  
\maketitle
\vspace{-1.5em}
\begin{abstract}
Quadratically regularized optimal transport (QOT) is an alternative to entropic regularization that yields sparse couplings and avoids numerical instabilities due to exponential scaling. From an optimization viewpoint, the dual QOT objective is concave but features a positive part function which prevents strong concavity and reduces smoothness of optimizers. Consequently, standard arguments for linear convergence of algorithms do not apply. In this paper, we nevertheless establish a quantitative curvature property for the QOT dual. Under mild assumptions covering both continuous and semi-discrete transport problems, we prove a local error bound and a Polyak--\L{}ojasiewicz (PL) inequality, with explicit constants depending only on the problem primitives. These results are obtained by functional-analytic techniques exploiting that near the optimum, the argument of the positive part function is positive on the interior of the support of the optimal coupling. As applications, we derive linear convergence of the gradient ascent, coordinate ascent, and coordinate gradient ascent algorithms on the dual problem, with explicit contraction rates.
\end{abstract}

\vspace{1em}

{\small

}

\section{Introduction}

Optimal transport provides a principled framework for comparing probability distributions, with ubiquitous applications from machine learning to economics. Given compactly supported probability measures $P$ and $Q$ on $\R^d$, the optimal transport problem with cost function $c:\R^d\times\R^d \to\R$ is 
\begin{equation}\label{otIntro}
  {\rm OT}(P,Q)=  \inf_{\pi\in \Pi(P,Q)} \int_{\R^d\times\R^d} c(x,y)\,d\pi(x,y),
\end{equation}
where $\Pi(P,Q)$ denotes the set of couplings (or transport plans), i.e., probability measures on $\R^d\times\R^d$ with marginals $(P,Q)$. In particular, the optimal value ${\rm OT}(P,Q)$ induces the $p$-Wasserstein distance when $c(x,y)=\|x-y\|^p$. As~\eqref{otIntro} is computationally and statistically challenging in high dimensions, the dominant approach in applications is to introduce a regularization term. Popularized by~\cite{Cuturi.2013.Neurips}, the most widely used is Kullback--Leibler (KL) divergence, leading to the \emph{entropic optimal transport (EOT)} problem
\begin{equation}\label{EOT}
    {\rm EOT}_{\eps}(P,Q)= \inf_{\pi\in \Pi(P,Q)} \int c(x,y)\,d\pi(x,y)
    +\eps{\rm KL}\big(\pi\vert P\otimes Q\big).
\end{equation}
Here $\eps>0$ is a parameter controlling the strength of the regularization. The associated dual problem is the maximization of
\begin{equation}\label{dual.EOT}
    \int \left(f(x)+ g(y)
    -\eps e^{\frac{f(x)+g(y)- c(x,y)}{\eps}}
    \right) d(P \otimes Q)(x,y)
\end{equation}
over functions $f,g:\R^d\to\R$.
The celebrated Sinkhorn algorithm can be interpreted as the (block) coordinate ascent algorithm on this objective. Two key features underlying the success of EOT and Sinkhorn's method are the strong concavity of the objective~\eqref{dual.EOT} and the fact that its optimizers $(f_*,g_*)$ are automatically as smooth as the cost~$c$. In particular, strong concavity directly yields linear convergence of Sinkhorn's algorithm~\cite{Carlier.2022.SIOPT}. At the same time, entropic regularization introduces structural distortions. The infinite derivative of $t \log t$ at zero entails that the optimal coupling of~\eqref{EOT} necessarily has full support (the same support as $P\otimes Q$), even though~\eqref{otIntro} is typically sparse. This ``overspreading'' induces blurring or bias in applications such as image processing~\cite{blondel18quadratic} or manifold learning~\cite{zhang.2023.manifoldlearningsparseregularised}. Moreover, weak regularization $\eps\ll1$ leads to numerical instability due to exponentially large and small values in the Sinkhorn updates~\cite{Schmitzer.19}. See, e.g., \cite{Peyre.Cuturi.2019.Book,Nutz.2021.LectureNotes} for further background. 

\medskip

The most popular alternative, first considered by \cite{Muzellec.2017.AAAI,blondel18quadratic,EssidSolomon.18}, replaces KL divergence by the squared $L^2$-norm (equivalently, $\chi^2$ divergence), leading to the \emph{quadratically regularized optimal transport (QOT)} problem
\begin{equation}\label{qotIntro}
  \QOT_\eps(P,Q)=  \inf_{\pi\in \Pi(P,Q)} \int c(x,y) d\pi(x,y) +\frac{\eps }{2}\left\| \frac{d \pi}{ d(P\otimes Q)}\right\|^2_{L^{2}(P\otimes Q)}.
\end{equation}
Its optimal value converges to ${\rm OT}(P,Q)$ as $\eps\to0$ at the rate $\eps^{2/(d+2)}$ \cite{Eckstein.Nutz.2023}. Unlike entropic regularization, the optimal coupling is sparse for small to moderate~$\eps$, as has been observed empirically in many works including \cite{blondel18quadratic,EssidSolomon.18, Lorenz.2019,BayraktarEckstein.2025.BJ} and recently established theoretically in~\cite{WieselXu.24,GonzalezSanzNutz2024.Scalar, Nutz.2024}. Moreover, exponentially large or small values do not appear in~QOT. A variety of computational approaches have been developed, including mirror gradient methods~\cite{Muzellec.2017.AAAI}, Newton-type
algorithms~\cite{EssidSolomon.18}, Gauss–Seidel (coordinate ascent) schemes~\cite{blondel18quadratic,Lorenz.2019}, cyclic projections and gradient methods~\cite{lorenz.2019.preprint}, as well as neural network methods parameterizing the dual domain \cite{EcksteinKupper.21, GeneveyEtAl.16, GulrajaniAhmedArjovskyDumoulinCourville.17, LiGenevayYurochkinSolomon.20, seguy2018large}.

From an optimization viewpoint, quadratic regularization leads to a markedly different geometry. The dual problem associated with~\eqref{qotIntro} is the maximization of the concave objective
\begin{align}\label{eq:dualobj-intro}
     \Gamma(f,g) = \int \left(f(x) + g(y) -\frac{1}{2\eps} \big(f(x)+ g(y)- c(x,y)\big)_+^{2} \right) d(  P \otimes   Q )(x,y).
\end{align}
Note that the positive part function $(t)_+=\max(t,0)$ destroys strong (and even strict) concavity: $\Gamma$ is linear in $f\oplus g$ as long as $f\oplus g-c\leq0$, where $(f\oplus g)(x,y):=f(x)+g(y)$. Moreover, optimizers have limited smoothness (not $\cC^{2}$ even if $c(x,y)=\|x-y\|^2$, cf.\ \cite{GonzalezSanzEcksteinNutz.25,Nutz.2024}). Thus, many of the arguments that are used in EOT do not apply. While well established in computational practice, QOT has long been regarded as difficult to study analytically, and it is fair to say that the theory is still in its infancy. However, two recent results suggest that QOT is better behaved than the aforementioned properties suggest: despite the lack of strong concavity and smoothness, \cite{GonzalezSanzDelBarrioNutz.25} established that QOT has parametric sample complexity (i.e., does not suffer from the same statistical curse of dimensionality as unregularized optimal transport), and \cite{GonzalezSanzNutzRiveros.25} showed that the gradient ascent algorithm for~\eqref{eq:dualobj-intro} converges linearly. These results hint that the geometry of the QOT dual may retain some features of strong concavity, at least locally around the optimum. The goal of the present work is to crystallize this intuition into a quantitative result.

Indeed, our main result provides an error bound and a Polyak--\L{}ojasiewicz (PL) inequality for the QOT dual~\eqref{eq:dualobj-intro} under mild assumptions, thus offering the first structural understanding for the optimization and statistical properties of QOT. Denoting by $(f_*,g_*)$ the (essentially unique) dual optimizers, \cref{th:PL-Gamma} states that
$$ \|f\oplus g-f_*\oplus g_*\|_{L^2(P\otimes Q)} \leq \gamma_{\eps}\max(\|f\oplus g-f_*\oplus g_*\|_{\infty} , \eps) \,\|{\rm D}\Gamma(f,g)\|_{L^2(P)\times L^2(Q)}$$
and
\begin{align*}
    \|{\rm D}\Gamma (f,g)\|_{L^{2}(P) \times L^{2}(Q)}^{2} 
    &\geq \frac{1}{\gamma_\eps  \max\left(\|f\oplus g - f_*\oplus g_*\|_\infty,\eps\right) }\big(\QOT_\eps(P,Q)- \Gamma (f,g) \big),%
\end{align*}
where ${\rm D}\Gamma$ denotes the gradient and $\gamma_\eps$ is an explicit constant. This error bound and PL inequality are local in the sense that their constants are uniform only for $\|f\oplus g - f_*\oplus g_*\|_\infty\leq\eps$, where $\eps$ is the (fixed) regularization parameter. Such a localization is clearly necessary, given the aforementioned linear structure of the dual on part of the domain. Our setting assumes that one marginal has compact and convex support (or more generally, compact connected Lipschitz support)  and a density that is bounded from above and below, whereas the second marginal is only assumed to be compactly supported and the transport cost~$c$ is any Lipschitz (or uniformly continuous) function. In particular, our setting covers both continuous and semi-discrete transport problems. 

It is well known that a PL inequality enables a host of desirable properties for optimization. In this paper, our applications focus on the convergence of several algorithms for the dual QOT problem (and hence also for the primal problem~\eqref{qotIntro}, since the dual optimizers yield the optimal coupling via the relation~\eqref{eq:primal-dual} below). Specifically, we show in \cref{coro:GD,coro:Implicit,coro:Explicit} that gradient ascent, coordinate ascent, and coordinate gradient ascent all converge linearly (for suitable step sizes), with explicit contraction constants. Of course, the PL inequality may also be useful for analyzing other algorithms. %
Further important applications, to be developed more fully in forthcoming works, include explicit bounds for the statistical sample complexity of QOT and quantitative stability for the dual (and hence also the primal)  QOT optimizers. Regarding the second application, we mention that first Lipschitz-stability results follow immediately along the lines of \cite[Section~4.4]{BonnansShapiro.00}, 
whereas stronger results follow with additional work (see \cite{GonzalezSanzNutz.26stability}). These implications of the PL inequality offer a marked improvement over the H\"older continuity of optimal couplings shown in \cite[Theorem~3.3]{BayraktarEckstein.2025.BJ}.

The basic intuition for our main result is the following. The dual objective~\eqref{eq:dualobj-intro} would be strongly concave (in $f\oplus g$) if it weren't for the positive part operator. The latter is active only on the complement of the set $\cS=\{(x,y): f(x)+ g(y) \geq c(x,y)\}$. When $(f,g)=(f_*,g_*)$ are dual optimizers, $\cS$ is the support of the optimal coupling, by way of~\eqref{eq:primal-dual}. While we expect the support to be sparse, the fact that it carries the solution also implies certain lower bounds. Thus, if $(f,g)$ is close to $(f_*,g_*)$, we know that there is a nontrivial set where the positive part operator is not active, and that set must be exploited as the source of curvature for our result. However, it is not straightforward to turn this intuition into a formal argument (in part, because we have very limited knowledge about the shape and size of the optimal support). The actual proof follows an original approach leveraging functional-analytic tools: We fix $(f,g)$ and consider the dual objective $\Gamma(f_t,g_t)$ along the interpolation $(f_t,g_t):=t(f,g) + (1-t)(f_*,g_*)$. The main step is to bound the second derivative $\partial_{tt}^2 \Gamma(f_t,g_t)$ from below, which corresponds to the coercivity of an integral operator whose kernel is $\1_{\{f_t(x) + g_t(y) \geq c(x,y)\}}$. For $t=0$, this set is the support of the optimal coupling, linking to the intuition sketched above. Establishing the lower bound boils down to bounding the minimum eigenvalue of an associated operator, which requires us to develop a tailored spectral analysis since the operator is not compact. 

As the PL inequality is local, drawing the corollaries for the algorithms requires uniform bounds for the iterates. These bounds are algorithm-specific, but once they are obtained, the linear convergences are direct consequences of the PL inequality. Of course, this analysis could be extended to variants of the discussed algorithms, for instance variable step sizes. As mentioned, linear convergence of gradient ascent was previously shown in \cite{GonzalezSanzNutzRiveros.25}. However, the contraction constant given there is not explicit (due to the proof containing a qualitative compactness argument). Moreover, the proof is not easily adapted to other algorithms such as coordinate ascent or coordinate gradient ascent since it relies on a symmetry property specific to the gradient ascent algorithm. The derivation through the PL inequality in the present work offers a systematic approach that covers a whole class of algorithms. Apart from~\cite{GonzalezSanzNutzRiveros.25}, we are not aware of previous results on linear convergence for continuous QOT.

\paragraph{Organization} \Cref{se:background} details the assumptions and collects relevant background results. The main result on error bound and PL inequality is \cref{th:PL-Gamma} in \cref{se:PL}, whereas the linear convergences of the algorithms are stated as \cref{coro:GD,coro:Implicit,coro:Explicit} in \cref{Section:linear-convergence}. The proof of the main result is given in \cref{Section-Proof-of-PL}, and the corollaries for the algorithms are derived in  \cref{se:proofs-of-linear}. \Cref{sec:conclusion} concludes with a summary and practical remarks. Appendix~\ref{app:uniformly-continuous-costs} extends the results from Lipschitz to uniformly continuous costs, and Appendix~\ref{app:PL-connected} relaxes the convexity assumption on the first marginal support.

\section{Main Results}\label{se:main}

\subsection{Problem statement and background}\label{se:background}

We first detail our assumptions on the marginals $P,Q$ and the cost $c$. Let $\cL_{d}$ denote the $d$-dimensional Lebesgue measure and $B_r(x)$ the open ball of center $x$ and radius $r>0$. We fix probability measures~$P$ and~$Q$ on~$\RR^{d}$ and impose the following conditions throughout.

\begin{assumption}[Marginals] \label{assumption:marginals} (a) The first marginal~$P$ has convex, compact support~$\Omega$ and admits a density $\rho:=dP/d\cL_{d}$ that is bounded away from zero and infinity on $\Omega$, i.e., there exist constants $0<\lambda_P<\Lambda_P<\infty$ such that 
$$ \lambda_P\leq \rho(x) \leq \Lambda_P \quad \text{for all } x\in \Omega.$$
(b) The second marginal $Q$ has compact support $\Omega'$.
\end{assumption}

The convexity condition can be relaxed to $\Omega$ being the closure of a bounded Lipschitz domain. All our results extend to that setting, at the expense of an additional constant that depends on the geometry of $\Omega$ and a more complicated proof. This extension is detailed in \cref{app:PL-connected}.

\begin{remark}[Lower bounds for ball measures]\label{rk:coneCondition} (a) Since $\Omega$ is compact and convex, it satisfies a uniform interior cone condition: for every $x\in \Omega$, there exists a convex cone $\mathcal{C}_x$  with vertex $x$, angle $\theta>0$ and height $h>0$ such that $\mathcal{C}_x\subset \Omega$, where~$\theta$ and~$h$ are independent of $x$ (e.g., \cite[p.\,12]{Grisvard.85}). Hence,  \cref{assumption:marginals} implies that there exists $\delta_P\in (0,1]$, depending only on~$\Omega$ and~$\lambda_P$, such that
$$ P(B_r(x)) \geq  \delta_P\, \min(r^d, 1) \quad \text{for all $r>0$ and $x\in \Omega$.} $$
(b) Since $\Omega'$ is compact, lower semicontinuity of $y\mapsto Q(B_r(y))$ implies
$$ \inf_{y\in \Omega' }Q(B_r(y))>0 \quad \text{for all $r>0$.} $$
\end{remark}

We note that $Q$ can be continuous or discrete, that is, our setup covers both continuous and semi-discrete optimal transport.

\begin{assumption}[Cost] \label{assumption:cost} The cost $c:\R^d\times \R^d\to \R$ is Lipschitz with constant $L>0$.      
\end{assumption}

Lipschitz-continuity can be relaxed to uniform continuity at the expense of a more complicated expression for the constant in the main results; see \cref{app:uniformly-continuous-costs}.

The quadratically regularized optimal transport (QOT) problem with regularization parameter $\eps>0$ is
\begin{equation}\label{eq:primal}
  \QOT_\eps(P,Q):=  \inf_{\pi\in \Pi(P,Q)} \int c(x,y) d\pi(x,y) +\frac{\eps }{2}\left\| \frac{d \pi}{ d(P\otimes Q)}\right\|^2_{L^{2}(P\otimes Q)}
\end{equation}
with the convention that the last term is $+\infty$ if $\pi \not\ll P\otimes Q$. The dual problem is 
\begin{equation}\label{eq:dual}
     \DUAL_\eps(P,Q)=\sup_{(f,g)\in L^2(P)\times L^2(Q)} \Gamma(f,g),
\end{equation}
where the dual objective function $\Gamma: L^{2}(P) \times L^{2}(Q) \to \RR$ is defined by
\begin{align}\label{eq:dualoperator}
     \Gamma(f,g) := \int \left(f(x) + g(y) -\frac{1}{2\eps} \big(f(x)+ g(y)- c(x,y)\big)_+^{2}\right) d(P\otimes Q)(x,y).
\end{align}
The gradient of $\Gamma$ at $(f,g)\in L^{2}(P) \times L^{2}(Q)$ is
\begin{equation}\label{eq:gradientGamma}
{\rm D}\Gamma\left(  \begin{array}{c}
    f\\
    g
\end{array} \right)
= \left(  \begin{array}{c}
    {\rm D}_1 \Gamma \\
    {\rm D}_2 \Gamma
\end{array} \right)  \left(  \begin{array}{c}
    f\\
    g
\end{array} \right)= \left(  \begin{array}{c}
     1-\frac{1}{\eps}\int{  \left(f(\cdot)+ g(y)- c(\cdot, y)\right)_+ } dQ(y)\\[.3mm]
      1-\frac{1}{\eps}\int{  \left(f(x)+ g(\cdot)- c(x, \cdot)\right)_+ } dP(x)
\end{array} \right)
\end{equation}
and its action on $(u,v) \in L^{2}(P) \times L^{2}(Q)$ is
$$ \langle  {\rm D}\Gamma(f,g),  (u,v) \rangle_{ L^{2}(P) \times L^{2}(Q) } = \int  (u\oplus v) \left(1- \frac{1}{\eps} \left(f\oplus g- c\right)_+\right) d(P\otimes Q). $$

Next, we recall a number of known properties. They do not require the full strength of \cref{assumption:marginals}, only that $\Omega,\Omega'$ are compact and~$\Omega$ is connected.

\begin{proposition}\label{pr:prelims}
    \begin{enumerate}
    \item
    The strong duality 
      $\QOT_\eps(P,Q)=\DUAL_\eps(P,Q)$
    holds. 
    \item
    The dual problem~\eqref{eq:dual} admits an optimizer $(f_*,g_*)\in L^2(P)\times L^2(Q)$. The functions $(f_*,g_*)$ are unique up to translation in the following sense: fixing one optimizer $(f_*,g_*)$, the set of all optimizers is given by $\{(f_*+a,g_*-a): a\in\R\}$. 
    \item We can choose versions $f_*:\Omega\to\R$ and $g_*:\Omega'\to\R$ that are $L$-Lipschitz. Those versions are fixed in all that follows. We call $(f_*,g_*)$ the \emph{potentials.}
    \item
    The potentials $(f_*,g_*)$ satisfy the first-order condition
    \begin{equation}
        \label{eq:FOC}
        \begin{cases}
             \int (f_*(x)+ g_*(y)-c(x,y))_+ dQ(y)=\eps \quad \text{for all }x\in\Omega,\\
             \int (f_*(x)+ g_*(y)-c(x,y))_+ dP(x)=\eps \quad \text{for all }y\in\Omega'.
        \end{cases}
    \end{equation}
    \item
    The primal problem~\eqref{eq:primal} has a unique solution $\pi_\eps\in\Pi(P,Q)$. It is related to the potentials by
    \begin{equation}
    \label{eq:primal-dual}
    \frac{d\pi_\eps}{d(P \otimes Q)}(x,y) = \frac{1}{\eps} \big( f_*(x) + g_*(y) - c(x,y) \big)_+.
\end{equation}
\end{enumerate}
\end{proposition}

We refer to the proof of \cite[Lemma~3.1]{GonzalezSanzNutzRiveros.25} for the uniqueness of the potentials and to \cite[Section~2]{Nutz.2024} for all other claims.
 
Given functions $f:\Omega\to\R$ and $g:\Omega'\to\R$, we denote $(f\oplus g)(x,y):=f(x)+g(y)$. The uniqueness of the potentials up to translation implies that $f_*\oplus g_*$ is uniquely determined. More generally, we observe that the dual objective~\eqref{eq:dualoperator} can be seen as a function of $f\oplus g$ rather than the two separate functions $f$ and $g$. This point of view is used for our main proofs, which consider the operator $\Phi(f\oplus g):=\Gamma(f,g)$ on the space $\Hil = \big\{f\oplus g: (f,g)\in L^{2}(P)\times L^{2}(Q)\big\} \subset  L^{2}(P\otimes Q)$. This is mathematically convenient and also leads to slightly tighter constants. On the other hand, algorithms (and scientists) often consider $f,g$ as individual functions that are updated separately, hence we take that point of view in the presentation of the main results below.

\subsection{PL inequality}\label{se:PL}

Our main result is a PL-type inequality for the dual objective~$\Gamma$ of~\eqref{eq:dualoperator}. We recall that  $(f_*,g_*)\in \cC(\Omega)\times\cC(\Omega')$ are fixed (but arbitrary) potentials and the constants $\lambda_P,\Lambda_P,\delta_P,L$ were defined in \cref{assumption:marginals}, \cref{rk:coneCondition} and \cref{assumption:cost}. 

\begin{theorem}[PL inequality for $\Gamma$]\label{th:PL-Gamma}
    For $(f,g)\in L^\infty(\Omega)\times L^\infty(\Omega')$, we have the error bound
    $$ \|f\oplus g-f_*\oplus g_*\|_{L^2(P\otimes Q)} \leq \gamma_{\eps}\max(\|f\oplus g-f_*\oplus g_*\|_{\infty} , \eps) \,\|{\rm D}\Gamma(f,g)\|_{L^2(P)\times L^2(Q)}$$
    and the PL inequality
    \begin{align}
        \|{\rm D}\Gamma (f,g)\|_{L^{2}(P) \times L^{2}(Q)}^{2} 
        &\geq \frac{1}{\gamma_\eps  \max\left(\|f\oplus g - f_*\oplus g_*\|_\infty,\eps\right) }\big(\QOT_\eps(P,Q)- \Gamma (f,g) \big) ,\label{eq:PLforGamma}
    \end{align}
    where 
    \begin{align}\label{eq:gamma-defn}
    \gamma_\eps=  16 \left(\delta_P^{-1}\max\left( \frac{8L}{\eps},1\right)^d\right) \frac{\Lambda_P^2}{\lambda_P^2}
\frac{\bigl(\lceil 8L\,{\rm diam}(\Omega)/\eps \rceil\bigr)^{d+2}}{\inf_{y\in\Omega'} Q\bigl(B_{\frac{\eps}{8L}}(y)\bigr)}.
    \end{align}
\end{theorem}

The proof is given in \cref{Section-Proof-of-PL}.

\subsection{Convergence of algorithms}\label{Section:linear-convergence}
Next, we apply the PL inequality to show that three natural algorithms for the dual problem~\eqref{eq:dual} converge linearly for suitable step sizes: gradient ascent, coordinate ascent, and coordinate gradient ascent. All algorithms have iterates $(f_n,g_n)$, and linear convergence means that the suboptimality gap $\Delta_n : = \QOT_\eps(P,Q)- \Gamma(f_n,g_n)\equiv \Gamma(f_*,g_*) - \Gamma(f_n,g_n)$ satisfies $\Delta_n \leq (1-q)^n \Delta_0$ for some $q\in (0,1)$ that is made explicit below. Numerical experiments for these algorithms can be found in~\cite{lorenz.2019.preprint} (they are omitted in the journal version~\cite{LorenzMahler.22}). The experiments suggest that all three algorithms are efficient and consistent in the examples considered, but no theoretical analysis was given. 

The proof of each corollary below is reported in \cref{se:proofs-of-linear} and combines three facts: the PL inequality (\cref{th:PL-Gamma}), the Lipschitz property of ${\rm D}\Gamma$ (\cref{le:lipschitzgradient}), and a uniform bound for $\|f_n\oplus g_n-f_*\oplus g_*\|_{\infty}$. The latter gives a uniform constant for the PL inequality and has a slightly different proof for each algorithm; after that, the argument for linear convergence is standard. As mentioned in the introduction, linear convergence of gradient ascent was previously shown in \cite{GonzalezSanzNutzRiveros.25}, but without explicit control of the contraction rate. Apart from that, we are not aware of any previous results on convergence rates of QOT algorithms. For the statements that follow, we recall that  $(f_*,g_*)\in \cC(\Omega)\times\cC(\Omega')$ are fixed (but arbitrary) potentials and the constant $\gamma_\eps$ defined in~\eqref{eq:gamma-defn}.

\subsubsection{Gradient ascent}

The  gradient ascent algorithm with step size $\eta>0$ is initialized at $(f_0,g_0)\in L^\infty(\Omega)\times L^\infty(\Omega')$ and defined by
\begin{equation}\label{eq:gradientascent}
 \left(  \begin{array}{c}
     f_{n+1}\\
       g_{n+1}
\end{array} \right) = \left(  \begin{array}{c}
     f_{n}\\
       g_{n}
\end{array} \right)+ \eta\,   {\rm D}\Gamma \left(  \begin{array}{c}
     f_{n}\\
       g_{n}
\end{array} \right) , \quad n\geq0,
\end{equation}
where the gradient ${\rm D}\Gamma$ has the explicit form stated in~\eqref{eq:gradientGamma}. It is well known that a PL inequality implies linear convergence of gradient ascent. Given the specific form of our (local) PL inequality~\eqref{eq:PLforGamma}, we only need to bound $\|f_n\oplus g_n-f_*\oplus g_*\|_{\infty}$ uniformly in~$n$. 

\begin{corollary}[Linear convergence of gradient ascent]\label{coro:GD}
    Let $(f_0,g_0)\in L^\infty(\Omega)\times L^\infty(\Omega')$ and define $(f_n,g_n)_{n\geq 1}$ by~\eqref{eq:gradientascent}. For every $\eta\in (0,\eps)$ and $n\geq1$,
    \begin{align}\label{eq:GDIterateBound}
      \|f_{n}\oplus g_{n} - f_*\oplus g_*\|_\infty \leq 2 \|f_{0}\oplus g_{0} - f_*\oplus g_*\|_\infty
    \end{align}
    and the suboptimality gap $\Delta_n : = \QOT_\eps(P,Q)- \Gamma(f_n,g_n)$ satisfies
     \begin{align*}
          \Delta_n \leq (1-q)^n \Delta_0\qquad\mbox{for}\qquad q:=
          \frac{\eta \left(1-\frac{\eta }{\eps} \right)} {\gamma_{\eps}\max(2\|f_{0}\oplus g_{0} - f_*\oplus g_*\|_\infty, \eps)}.
     \end{align*}
     Moreover,
     \begin{align*}
         \|f_n\oplus g_n - f_*\oplus g_*\|_{L^2(P\otimes Q)}^2
         \leq \frac{\gamma_\eps^2\max(2\|f_{0}\oplus g_{0} - f_*\oplus g_*\|_\infty, \eps)^2}{\eta \left(1-\frac{\eta }{\eps} \right)}\,\Delta_0\,(1-q)^n.
     \end{align*}
     
\end{corollary}

\subsubsection{Coordinate ascent}
The coordinate ascent algorithm is an implicit algorithm which is initialized at $g_0\in L^\infty(\Omega')$ and iteratively optimizes the coordinates of $\Gamma$,
\begin{equation}
    \label{eq:implicit}
    f_{n} := \argmax_{f} \Gamma(f,g_n), \qquad g_{n+1} := \argmax_{g} \Gamma(f_{n},g), \qquad n\geq 0 .
\end{equation}
This is equivalent to iteratively solving the first-order conditions for $n\geq0$:
    \begin{align}
            \text{define $f_n(x)$ via} \quad \eps &= \int \big( f_{n}(x) + g_{n}(y) - c(x,y) \big)_+ \, dQ(y), \label{eq:implicitalgo-f} \\
            \text{then $g_{n+1}(y)$ via} \quad \eps &= \int \big( f_{n}(x) + g_{n+1}(y) - c(x,y) \big)_+ \, dP(x).\label{eq:implicitalgo-g}
    \end{align}
These equations imply that the argmax is unique and a bounded function (by the same arguments as \cite[Lemmas 2.4, 2.5]{Nutz.2024}), so that the iterates are uniquely defined.

We remark that the coordinate ascent algorithm becomes Sinkhorn's algorithm when the quadratic regularization is replaced by KL divergence (i.e., entropic optimal transport). In that case, the above equations are solved explicitly as the properties of the exponential function allow one to take $f_n$ and $g_{n+1}$ out of the integral. For quadratic regularization, the definition remains implicit. In \cite{LorenzMahler.22}, methods such as line search were proposed to solve the implicit problems. 

\begin{corollary}[Linear convergence of coordinate ascent]\label{coro:Implicit}
 Let $g_0\in L^\infty(\Omega')$ and define $(f_n,g_n)_{n\geq 0}$ by~\eqref{eq:implicit}. Then
 \begin{align}\label{eq:implicitIteratesBound}
 \| g_{n+1}  - g_{*} \|_{\infty} \leq  \| f_{n} - f_{*} \|_{\infty}\leq \| g_{n}  - g_{*} \|_{\infty} \leq \dots \leq \| g_0- g_* \|_\infty   
 \end{align}
 and the suboptimality gap $\Delta_n : = \QOT_\eps(P,Q)- \Gamma(f_n,g_n)$ satisfies
 \begin{align*}
    \Delta_n \leq (1-q)^n \Delta_0\qquad\mbox{for}\qquad 
    q:= \frac{\eps} {2\gamma_\eps\max(2\|g_{0} - g_*\|_\infty , \eps)}.
 \end{align*}
 Moreover,
 \begin{align*}
    \|f_n\oplus g_n - f_*\oplus g_*\|_{L^2(P\otimes Q)}^2
    \leq \frac{2\gamma_\eps^2\max(2\|g_{0} - g_*\|_\infty , \eps)^2}{\eps}\,\Delta_0\,(1-q)^n.
 \end{align*}
 
\end{corollary}

In the definition of~$q$, the constant $\|g_{0} - g_*\|_\infty$ can be replaced by $\inf_{a\in\R}\|g_{0} - g_*+a\|_\infty$ since the corollary holds for any pair of potentials, including  $(f_*+a,g_*-a)$.

\subsubsection{Coordinate gradient ascent}
The coordinate gradient ascent algorithm with step size $\eta>0$ is initialized at $(f_0,g_0)\in L^\infty(\Omega)\times L^\infty(\Omega')$ and defined by
\begin{equation}\label{Explicit-coordinate}
 f_{n+1} :=  f_{n} + \eta\,   {\rm D}_1\Gamma \left(\begin{array}{c}
    f_{n}\\
    g_{n}
\end{array} \right), \qquad g_{n+1} :=  g_{n} +  \eta\,  {\rm D}_2\Gamma \left(  \begin{array}{c}
    f_{n+1}\\
    g_{n}
\end{array} \right), \qquad n\geq0.
\end{equation}

\begin{corollary}[Linear convergence of coordinate gradient ascent]\label{coro:Explicit}
Let $(f_0,g_0)\in L^\infty(\Omega)\times L^\infty(\Omega')$ and define $(f_n,g_n)_{n\geq 1}$ by \cref{Explicit-coordinate}. For every $\eta\in (0,\eps/\sqrt{2})$,
\begin{align}\label{eq:ExplicitIterateBound}
    \max ( \|f_{n+1} - f_{*}\|_{\infty}, \|g_{n+1} - g_{*}\|_{\infty}) 
    &\leq \max ( \|f_{n} - f_{*}\|_{\infty}, \|g_{n} - g_{*}\|_{\infty})
\end{align}
    and the suboptimality gap $\Delta_n : = \QOT_\eps(P,Q)- \Gamma(f_n,g_n)$ satisfies
     \begin{align*}
          \Delta_n \leq (1-q)^n \Delta_0\qquad\mbox{for}\qquad 
          q:= \frac{\eta \left(1- \frac{\eta}{2\eps}\right) } {2\gamma_\eps\max(2\|f_{0}\oplus g_{0} - f_*\oplus g_*\|_\infty , \eps)}.
     \end{align*}
     Moreover,
     \begin{align*}
         \|f_n\oplus g_n - f_*\oplus g_*\|_{L^2(P\otimes Q)}^2
         \leq \frac{2\gamma_\eps^2\max(2\|f_{0}\oplus g_{0} - f_*\oplus g_*\|_\infty , \eps)^2}{\eta \left(1- \frac{\eta}{2\eps}\right)}\,\Delta_0\,(1-q)^n.
     \end{align*}
     
\end{corollary}

\section{Proof of the PL inequality}
\label{Section-Proof-of-PL}
Consider the linear subspace
$$
  \Hil = \big\{u\oplus v: (u,v)\in L^{2}(P)\times L^{2}(Q)\big\} \subset  L^{2}(P\otimes Q),
$$
which is closed by the argument in \cite[p.\,370]{RuschendorfThomsen.93}. Hence, $\Hil$ is a Hilbert space with the induced inner product $\langle \cdot,\cdot\rangle_{\Hil} = \langle \cdot,\cdot\rangle_{L^{2}(P\otimes Q)}$. We define the operator $\Phi: \Hil \to \RR$ by $\Phi(f \oplus g) = \Gamma(f,g)$, that is,
\begin{align}\label{eq:Phi}
     \Phi(f \oplus g) := \int \left(f(x) + g(y) -\frac{1}{2\eps} \left(f(x)+ g(y)- c(x,y)\right)_+^{2}\right) d(  P \otimes   Q )(x,y).
\end{align}
Working with $\Phi$ and $\Hil$ will be convenient as the direct sum removes the non-uniqueness of the potentials. The relation between the gradients of $\Phi$ and $\Gamma$ will be detailed in \cref{le:translationPhiToGamma}.

We fix $(f,g)\in L^\infty(\Omega)\times L^\infty(\Omega')$ in addition to the direct sum $f_* \oplus g_*$ of the potentials, and denote by $\phi$ the function 
$$ [0,1]\ni t\mapsto \phi(t)= - \Phi \left( f_{t} \oplus g_{t} \right), \quad \mbox{where} \quad f_{t} \oplus g_{t} := \left( (1-t)f_*+  tf \right) \oplus \left( (1-t)g_*+  tg \right).$$  
Note that $\phi$ is non-decreasing, convex, $\mathcal{C}^{1,1}$ and attains its minimum at $0$. Moreover, 
\begin{align}\label{eq:phiFirstDeriv}
\phi'(t)
= \int  ((f_*-f) \oplus (g_*-g)) \left(1- \frac{1}{\eps}(f_t\oplus g_t -c)_+\right) d(P\otimes Q),
\end{align}
and for a.e.~$t\in [0,1]$, 
\begin{align}\label{eq:phiSecondDeriv}
\phi''(t)= \frac{1}{\eps}\int_{f_t\oplus g_t\geq c}  ((f_*-f) \oplus (g_*-g))^2  d(P\otimes Q).
\end{align}
Here the differentiation under the integral can be justified by combining the proof of~\cite[Theorem~2.27]{Folland.99} and the a.e.\ differentiability of $s\mapsto (s)_+$.  
In particular, we have the expansion
\begin{align}
   \phi(t)&= \phi(0)+ \int_{0}^t \phi'(s) ds 
   =\phi(0)+ t\phi'(0)+  \int_{0}^t \int_0^s \phi''(r) d r ds \nonumber\\
   &=  \phi(0)+ \int_{0}^t \int_0^s \phi''(r) d r ds, \label{development-Phi}
\end{align}
where in the last step, we used~\eqref{eq:phiFirstDeriv} and~\eqref{eq:primal-dual} and the fact that $\pi_\eps\in\Pi(P,Q)$ yields
\begin{align*}
\phi'(0)
= \int  ((f_*-f) \oplus (g_*-g)) [d(P\otimes Q) - d\pi_\eps]=0.
\end{align*}   
In what follows, we set
\begin{align}\label{eq:def-C-r0}
C_{f,g}:=\| f_*\oplus g_* - f\oplus g\|_\infty, \qquad r_0:=\min\bigl(\tfrac{\eps}{2C_{f,g}},1\bigr)
\end{align}
and show in \cref{Theorem:Bound-double} that $\|(f_*-f) \oplus (g_*-g)\|_{L^2(P\otimes Q)}^2 \leq \frac{\eps}{\beta_\eps}\phi''(r)$ for a.e.\  $r\in[0,r_0]$, for a certain constant~$\beta_\eps$. Note that $r_0$ depends on $f,g$ only through $C_{f,g}$; this will be important below to obtain an inequality for arbitrary $f,g$. We mention that the choice of such~$r_0$ is inspired by an argument of~\cite{Stromme.24} in the context of entropic statistical optimal transport. Once \cref{Theorem:Bound-double} is shown, the following theorem will be deduced via~\eqref{eq:phiSecondDeriv}.

\begin{theorem}[PL inequality for $\Phi$]\label{th:PL-Phi} 
    The dual objective $\Phi:\Hil\to\R$ satisfies the error bound   
    $$ \|f\oplus g-f_*\oplus g_*\|_{L^2(P\otimes Q)} \leq \gamma_\eps  \max\left(\|f\oplus g-f_*\oplus g_*\|_{\infty},\eps\right)  \|{\rm D}\Phi(f \oplus g)\|_{L^2(P\otimes Q)}$$
    and the 
    PL inequality
    \begin{align}\label{eq:PLforPhi}
    \|{\rm D}\Phi(f \oplus g)\|_{L^2(P\otimes Q)}^2 \geq \frac{1}{\gamma_\eps  \max\left(\|f\oplus g-f_*\oplus g_*\|_{\infty},\eps\right)} \big(\QOT_\eps(P,Q)- \Phi(f \oplus g) \big)
    \end{align}
    for all $(f,g)\in L^\infty(\Omega)\times L^\infty(\Omega')$, where
    $$ \gamma_\eps= 16 \left(\delta_P^{-1}\max\left( \frac{8L}{\eps},1\right)^d\right) \frac{\Lambda_P^2}{\lambda_P^2}
\frac{\bigl(\lceil 8L\,{\rm diam}(\Omega)/\eps \rceil\bigr)^{d+2}}{\inf_{y\in\Omega'} Q\bigl(B_{\frac{\eps}{8L}}(y)\bigr)}.$$
\end{theorem}

In this statement, ${\rm D}\Phi(f \oplus g)\in\Hil$ denotes the $\Hil$-gradient of $\Phi$. Its explicit form is detailed in \cref{le:translationPhiToGamma} below.

\subsection{Coordinate-wise coercivity of $\phi''$}
With a view toward the second derivative~\eqref{eq:phiSecondDeriv}, the goal of this subsection is the coercivity of the nonnegative bilinear form  $$ (u\oplus v,  \tilde{u}\oplus \tilde{v})\mapsto \int_{f_r\oplus g_r\geq c}  (u\oplus v)(\tilde{u}\oplus \tilde{v})  d(P\otimes Q)$$ with respect to the first coordinate, for sufficiently small $r\in[0,1]$.
\begin{lemma}[Coordinate-wise coercivity]\label{lemma:PL-forVariance}
    Let $r_0$ be as in~\eqref{eq:def-C-r0}. For all $r\leq r_0$ and $(u,v)\in L^{2}(P)\times L^2(Q)$, 
$$
\int_{f_r\oplus g_r\geq c}  (u\oplus v)^2  d(P\otimes Q) 
\geq 
\alpha
\,\Var_P(u), \quad \alpha := \frac{\lambda_P^2}{\Lambda_P^2}
\frac{\inf_{y\in\Omega'} Q\bigl(B_{\frac{\eps}{8L}}(y)\bigr)}{\bigl(\lceil 8L\,{\rm diam}(\Omega)/\eps \rceil\bigr)^{d+2}},
$$
where $\Var_P(u) = \int [u(z)-\int u(x)\, dP(x)]^2\,dP(z)$ denotes the variance of~$u$ under $P$.
   \end{lemma}

The proof of \cref{lemma:PL-forVariance} uses the following auxiliary result, which is inspired by the Bourgain–Brezis–Mironescu limits for Sobolev energies (e.g., \cite[Theorem~1]{BourgainBrezisMironescu.01}).
   \begin{lemma}\label{Lemma:Bound-energy}
       For any convex and bounded set $\Omega\subset\R^d$, $\varrho>0$ and $h\in L^2(\Omega)$, the operator
       $$ \mathcal{E}_\varrho (h) := \int_\Omega \int_{B_\varrho(x) \cap \Omega }  (h(x)-h(z))^2  dz dx$$
       satisfies
       $$ \int_\Omega \int_\Omega (h(x)-h(z))^2  d x dz \leq  \left(\left\lceil{{\rm diam}(\Omega)/\varrho} \right\rceil \right)^{d+2} \mathcal{E}_\varrho (h).$$
   \end{lemma}
   \begin{proof} 
       Fix $x,z\in\Omega$ and set $x_k= x+ \frac{k}{m}(z-x)$ where $m=\left\lceil{{\rm diam}(\Omega)/\varrho} \right\rceil$. Then 
    \begin{align*}
        (h(x)-h(z))^2 & =   \left(\sum_{k=0}^{m-1} h(x_{k})-h(x_{k+1})\right)^2
        \leq m  \sum_{k=0}^{m-1} \left( h(x_{k+1})-h(x_{k})\right)^2
    \end{align*}
    and integrating this inequality gives
    \begin{multline}\label{eq:proof-bound-energy}
        \int_\Omega \int_\Omega (h(x)-h(z))^2  dz dx  \\
        \leq m\sum_{k=0}^{m-1} \underbrace{ \int_\Omega \int_\Omega \left( h\left( x+ \frac{k+1}{m}(z-x)\right)-h\left( x+ \frac{k}{m}(z-x)\right)\right)^2 dz dx }_{=:I_k} .
    \end{multline}
For fixed $k$, we use the change of variables
$$(u,v)=T(x,z)= \left(x+ \frac{k+1}{m}(z-x), x+ \frac{k}{m}(z-x)\right)$$ 
with
    \begin{align*}
{\rm D}T(x,z) &= A \otimes \mathbbm{I}_d,
\quad
A := \frac{1}{m}
\begin{pmatrix}
m-k-1 & k+1\\
m-k   & k
\end{pmatrix},\\
\det {\rm D}T(x,z)
&= \det(A\otimes \mathbbm{I}_d)
= \det(A)^d
= \left(-\frac{1}{m}\right)^d
\end{align*}
to get 
$$ I_k = m^{d} \int_{ T(\Omega\times \Omega)} (h(u)-h(v))^2 du dv.$$
We have $T(\Omega\times \Omega)\subset \Omega\times \Omega$ by convexity of $\Omega$. Noting that $(u,v)=T(x,z)$ satisfy $u-v=(z-x)/m$, we also have $\|u-v\|\leq {\rm diam}(\Omega)/m\leq \varrho$ for all $x,z\in\Omega$. As a  consequence, 
$ I_k\leq m^{d} \mathcal{E}_\varrho (h)$ and thus~\eqref{eq:proof-bound-energy} yields
$$ \int_\Omega \int_\Omega (h(x)-h(z))^2  dz dx \leq m^{d+2} \mathcal{E}_\varrho (h) $$
as claimed.
\end{proof}

We can now show \cref{lemma:PL-forVariance}. 

\begin{proof}[Proof of \cref{lemma:PL-forVariance}]
  Fix $r\leq r_0$ and $(u,v)\in L^2(P)\times L^2(Q)$.
  For each $y\in\Omega'$ we define the section $\cT^\eps_{y}= \{ x: f_*(x)+g_*(y)-c(x,y)\geq  \frac{\eps}{2}\}$ of the set $\{f_*\oplus g_*-c\geq \frac{\eps}{2}\}$. As $r\leq r_0$ implies $\{f_r\oplus g_r\geq c\}\supset\{f_*\oplus g_*-c\geq \frac{\eps}{2}\}$, we have
  \begin{align*}
    \cI:=\int_{f_r\oplus g_r\geq c}  (u\oplus v)^2  d(P\otimes Q)
    & \geq  \int_{f_*\oplus g_*-c\geq \frac{\eps}{2}}  (u\oplus v)^2  d(P\otimes Q)\\
    &= \int_{\Omega'} \int_{\cT^\eps_y}   (u(x)+v(y))^2  dP(x) dQ(y)\\
    &\geq \int_{\Omega'}  \min_{s\in \R}\int_{\cT^\eps_y}   (u(x)+s)^2  dP(x) dQ(y).
  \end{align*}
For fixed~$y$, the minimality property of the variance of $u$ under the probability measure $P(dx)/P(\cT^\eps_y)$ yields
\[
\min_{s \in \mathbb R} \int_{\cT^\eps_y} (u(x)+s)^2 \, dP(x)
= \frac{1}{2 P(\cT^\eps_y)}
   \int_{\cT^\eps_y} \int_{\cT^\eps_y} (u(x)-u(z))^2 \, dP(x)dP(z).
\]
Using also $P(\cT^\eps_y)\leq1$,
we conclude that 
\begin{align}\label{eq:proof_for_variance_step1}
    \cI
    & \geq \frac{1}{2}\int_{\Omega'} \int_{\cT^\eps_y} \int_{\cT^\eps_y}   (u(x)-u(z))^2  dP(x) dP(z) dQ(y).
\end{align}

Choose a measurable selector
\[
\Omega\ni x \mapsto y(x) \in \argmax_{y\in\Omega'} \,(f_*(x)+g_*(y)-c(x,y) )_+
\]
and set
\[
m_x = \max_{y\in\Omega'}\,(f_*(x)+g_*(y)-c(x,y))_+.
\]
Recalling~\eqref{eq:FOC}, we then have
\[
m_x \geq \int_{\Omega'} (f_*(x)+g_*(y)-c(x,y))_+ \, dQ(y) = \eps.
\]
Since $f_*$, $g_*$ and $c$ are $L$-Lipschitz, we deduce
\begin{align*}
f_*(z)+g_*(y)-c(z,y)
&\ge f_*(x)+g_*(y(x))-c(x,y(x))
      - 2L\|x-z\| - 2L\|y(x)-y\| \\
&= m_x - 2L\|x-z\| - 2L\|y(x)-y\| \\
&\ge \eps - 2L\|x-z\| - 2L\|y(x)-y\|
\end{align*}
and hence
\[
f_*(z)+g_*(y)-c(z,y) - \frac{\eps}{2}
\ge
\frac{\eps}{2} - 2L\|x-z\| - 2L\|y(x)-y\|.
\]
Let $\varrho := \frac{\eps}{8L}$ and $x\in\Omega$.  
If $z\in B_\varrho(x)$ and $y\in B_\varrho(y(x))$, then
\begin{align}\label{eq:forIndicatorRel}
f_*(z)+g_*(y)-c(z,y) - \frac{\eps}{2}
\ge
\frac{\eps}{2} - 2L\varrho - 2L\varrho
= \frac{\eps}{2} - \frac{\eps}{4} - \frac{\eps}{4}
= 0,
\end{align}
meaning that $z\in \cT^\eps_y$. This applies in particular to $z=x$, so $x\in\cT^\eps_y$.
Equivalently, for all $x,z\in\Omega$ and $y\in\Omega'$, we have
\[
\mathbf 1_{B_\varrho(y(x))}(y)\,\mathbf 1_{B_\varrho(x)}(z)
\le
\mathbf 1_{\cT^\eps_y}(x)\,\mathbf 1_{\cT^\eps_y}(z)\,
\mathbf 1_{B_\varrho(y(x))}(y).
\]
Applying this in~\eqref{eq:proof_for_variance_step1} gives
\begin{align*}
\cI
&\ge
\frac12 \int_{\Omega'} \int_{\Omega}\int_{\Omega}
\mathbf 1_{\cT^\eps_y}(x)\,\mathbf 1_{\cT^\eps_y}(z)\,
      (u(x)-u(z))^2 \, dP(x)\,dP(z)\,dQ(y) \\
&\ge
\frac12 \int_{\Omega'} \int_{\Omega}\int_{\Omega}
\mathbf 1_{B_\varrho(y(x))}(y)\,\mathbf 1_{B_\varrho(x)}(z)\,
      (u(x)-u(z))^2 \, dP(x)\,dP(z)\,dQ(y) \\
&=\frac12 \int_{\Omega} \left[
     Q\bigl(B_\varrho(y(x))\bigr)
     \int_{B_\varrho(x)} (u(x)-u(z))^2 \, dP(z)
\right] dP(x) \\
&\ge
\frac12 \Bigl(\,\inf_{y\in\Omega'} Q(B_\varrho(y))\Bigr)
       \int_{\Omega}\int_{B_\varrho(x)} (u(x)-u(z))^2 \, dP(z)\,dP(x).
\end{align*}
Recall from \cref{assumption:marginals} that the density $\rho$ of $P$ satisfies $\lambda_P\leq \rho \leq\Lambda_P$ on~$\Omega$, so that 
\[
\int_{\Omega}\int_{B_{\varrho}(x)} (u(x)-u(z))^2 \, dP(z)\,dP(x)
\ge
\lambda_P^2 \, \mathcal E_{\varrho}(u).
\]
By Lemma~\ref{Lemma:Bound-energy} applied with $\varrho=\frac{\eps}{8L}$,
\[
\mathcal E_{\varrho}(u)
\ge
\frac{1}{\bigl(\lceil {\rm diam}(\Omega)/\varrho \rceil\bigr)^{d+2}}
\int_\Omega \int_\Omega (u(x)-u(z))^2\,dx\,dz.
\]
Moreover,
\[
\int_{\Omega}\int_{\Omega} (u(x)-u(z))^2 \, dx\,dz
\ge
\frac{1}{\Lambda_P^2} \int_{\Omega}\int_{\Omega} (u(x)-u(z))^2 \, dP(x)\,dP(z)
= \frac{2}{\Lambda_P^2} \Var_P(u).
\]
Combining the last four displays,
\begin{align*}
\cI
&\ge
\Bigl(\inf_{y\in\Omega'} Q(B_{\varrho}(y))\Bigr)
\frac{\lambda_P^2}{\Lambda_P^2}
\frac{1}{\bigl(\lceil {\rm diam}(\Omega)/\varrho \rceil\bigr)^{d+2}}
\,\Var_P(u),
\end{align*}
which gives the claim after recalling that $\varrho = \frac{\eps}{8L}$.
\end{proof}

\begin{remark}\label{rk:ballContained}
  For a given $y\in\Omega'$, applying the argument around~\eqref{eq:forIndicatorRel} with $x$ replaced by $x(y) \in \argmax_{x\in\Omega} (f_*(x)+g_*(y)-c(x,y) )_+$ and $y(x)$ replaced by $y$ yields that $z\in\cT^\eps_y$ for any $z\in B_\varrho(x(y))\cap\Omega$. That is, $(B_\varrho(x(y))\cap\Omega)\subset\cT^\eps_y$. Symmetrically, we have for any $x\in\Omega$ that $(B_\varrho(y(x))\cap\Omega')\subset\cS^\eps_{x}$ for the $x$-section $\cS^\eps_{x}:=\{ y: f_*(x)+g_*(y)-c(x,y)\geq \frac{\eps}{2}\}$. Recalling  that $\varrho = \frac{\eps}{8L}$, we have in particular
  \begin{align}\label{eq:ballContained}
    \inf_{y\in\Omega'}Q\bigl(B_{\frac{\eps}{8L}}(y)\bigr)\leq Q(\cS^\eps_{x}), \qquad 
    \inf_{x\in\Omega}P\bigl(B_{\frac{\eps}{8L}}(x)\bigr)\leq P(\cT^\eps_{y}).
  \end{align}
\end{remark}

\subsection{Uniform coercivity of $\phi''$} 
Fix $r\in[0,1]$. By the Riesz representation theorem, there exists a unique bounded linear operator $\mathbb{M}:\Hil\to\Hil$ such that    
\begin{align}\label{eq:M}
\langle \tilde{u}\oplus \tilde{v},  \mathbb{M} (u\oplus v)\rangle_{\Hil}=    \int_{f_r\oplus g_r\geq c}  (\tilde{u}\oplus \tilde{v})(u\oplus v)  d(P\otimes Q)
\end{align}
for all $(u\oplus v),(\tilde{u}\oplus \tilde{v})\in\Hil$.
Denote by $\cS_{r,x}$ and $\cT_{r,y}$ the sections of the set $\cR_r:=\{f_r\oplus g_r\geq c\}$,
$$
\cS_{r,x} =\{y\in\Omega':  f_r(x)+ g_r(y)\geq c(x,y)\}, \qquad \cT_{r,y} =\{x\in\Omega:  f_r(x)+ g_r(y)\geq c(x,y)\}.
$$
Moreover, define 
$$
  m(u\oplus v) := \int_{f_r\oplus g_r\geq c} u\oplus v \,d(P\otimes Q) \;\in\R.
$$
\begin{lemma}\label{le:M-explicit}
The operator $\mathbb{M}:\Hil\to\Hil$ has the explicit representation
\begin{align}\label{eq:M-explicit}
\mathbb{M} (u\oplus v) = \mathbb{M}_{1}(u\oplus v) \oplus \mathbb{M}_{2}(u\oplus v),
\end{align}
where 
\begin{align*}
  \mathbb{M}_{1}(u\oplus v) &:= u Q(\cS_{r,(\cdot)}) + \int_{\cS_{r,(\cdot)}} v\,dQ - \frac{m(u\oplus v)}{2} \quad \in L^2(P),\\
  \mathbb{M}_{2}(u\oplus v) &:= v P(\cT_{r,(\cdot)}) + \int_{\cT_{r,(\cdot)}} u\,dP- \frac{m(u\oplus v)}{2} \quad\in L^2(Q).
\end{align*}
\end{lemma} 

\begin{proof}
    The definition~\eqref{eq:M} can be stated as
    \[
    \langle \tilde{w}, \mathbb{M} w \rangle_{L^2(P\otimes Q)}
    = \langle \tilde{w}, \mathbf{1}_{\cR_r} w \rangle_{L^2(P\otimes Q)}
\quad\text{for all }\tilde{w},w\in\Hil.
    \]
    This implies that $\mathbb{M} w$ is the orthogonal projection of $\mathbf{1}_{\cR_r} w \in L^2(P\otimes Q)$ onto the subspace $\Hil\subset L^2(P\otimes Q)$. For arbitrary $h\in L^2(P\otimes Q)$, define
\begin{align*}
h_P(x) &:= \int h(x,y)\,dQ(y), \qquad x\in\Omega,\\
h_Q(y) &:= \int h(x,y)\,dP(x), \qquad y\in\Omega',\\
\bar h &:= \int h(x,y)\,d(P\otimes Q)(x,y).
\end{align*}
We verify below that the orthogonal projection of $h\in L^2(P\otimes Q)$ onto~$\Hil$ is
\[
\proj h := h_P\oplus h_Q - \bar h. 
\]
Specializing to $h:=\mathbf{1}_{\cR_r}(u\oplus v)$, we have 
\[
\proj h = \left(h_P-\frac12\bar h\right)\oplus \left(h_Q-\frac12\bar h\right) =\mathbb{M}_1(u\oplus v)\oplus \mathbb{M}_2(u\oplus v),
\]
showing the claim.

It remains to prove that $\proj h$ is indeed the orthogonal projection. Given $h\in L^2(P\otimes Q)$, clearly $h_P\in L^2(P)$, $h_Q\in L^2(Q)$ and $\bar h\in\R$, so that $\proj h\in\Hil$.  To complete the proof, we need to show for all $\tilde u\oplus\tilde v\in\Hil$ that
\begin{align*}
0
&= \big\langle \tilde u\oplus\tilde v,\;h - \proj h\big\rangle_{L^2(P\otimes Q)}\\
&= \int_{\Omega\times\Omega'} (\tilde u(x)+\tilde v(y))
\big(h(x,y) - h_P(x) - h_Q(y) + \bar h\big)\,dP(x)dQ(y)
=: I_1 + I_2,
\end{align*}
where 
\begin{align*}
I_1
&:= \int \tilde u(x)\big(h(x,y) - h_P(x) - h_Q(y) + \bar h\big)\,dP(x)dQ(y)
\end{align*}
and $I_2$ is defined analogously with $\tilde v$ instead of $\tilde u$. Next, we show that $I_1=I_2=0$. 
By Fubini's theorem,
\[
\int \tilde u(x)h(x,y)\,dP(x)dQ(y)
= \int \tilde u(x)\left(\int h(x,y)\,dQ(y)\right)dP(x)
= \int \tilde u(x)h_P(x)\,dP(x).
\]
Moreover,
\begin{align*}
\int \tilde u(x)h_Q(y)\,dP(x)dQ(y)
&= \left(\int \tilde u(x)\,dP(x)\right)\left(\int h_Q(y)\,dQ(y)\right)
= \bar h\int \tilde u(x)\,dP(x).
\end{align*}
Expanding $I_1$ into four integrals, this shows that the first two and the last two integrals cancel, so that $I_1=0$. Similarly, $I_2=0$, completing the proof.
\end{proof}

The next lemma is an important technical step for the derivation of our main result. While $\mathbb{M}$ is not compact, the lemma uses the explicit form of $\mathbb{M}$ to establish that the minimum Rayleigh quotient is an eigenvalue. The proof technique is adapted from \cite[Proposition~4.1]{GonzalezSanzNutzRiveros.25}.

\begin{lemma}\label{lemma:self-adj}
    For any $r\in[0,1]$, the operator $\mathbb{M}:\Hil\to\Hil$ is bounded, self-adjoint and positive. Let
    \begin{equation}\label{eq:lambda0defn} \lambda_0 := \inf_{ \|u\oplus v\|_{\Hil}=1}\langle u\oplus v,  \mathbb{M} (u\oplus v)\rangle_{\Hil}
    \end{equation}
    and suppose there exists a constant $\kappa_{0}>0$ such that 
\begin{align}\label{eq:greater-lambda}
  Q(\cS_{r,(\cdot)})- \lambda_0 \geq \kappa_{0} \quad P\text{-a.s.}, \qquad P(\cT_{r,(\cdot)})- \lambda_0 \geq \kappa_{0} \quad Q\text{-a.s.}
\end{align} 
		Then $\lambda_0$ is an eigenvalue of~$\mathbb{M}$. As a consequence, the infimum in~\eqref{eq:lambda0defn} is attained and $\lambda_0$ is the smallest eigenvalue of~$\mathbb{M}$.
\end{lemma}

\begin{proof}
\emph{Step~1.} We readily see from the definition~\eqref{eq:M} that $\mathbb{M}$ is bounded, self-adjoint and positive on the Hilbert space $\Hil$. For the remainder of the proof we  abbreviate $\lambda:=\lambda_0$, $\langle \cdot,\cdot\rangle=\langle \cdot,\cdot\rangle_{\Hil}$ and $\|\cdot\|=\|\cdot\|_{\Hil}$, whereas $\|\cdot\|_{\op}$ denotes the operator norm. Noting
\[
  0 \le \langle w,\mathbb{M} w\rangle
  = \int_{\{f_r\oplus g_r\ge c\}} w^2\,d(P\otimes Q) 
  \le \int w^2\,d(P\otimes Q)
  = \|w\|^{2},
\]
we have 
$
  0 \le \mathbb{M}\le \Id,
$
where $\Id$ denotes the identity. 
Define the linear operator $\mathbb{L} : \Hil\to \Hil$ by
\[
  \mathbb{L} := \Id - \mathbb{M}.
\]
Then $\mathbb{L}$ is also bounded, self-adjoint and positive, so that its operator norm can be expressed via the Rayleigh quotient as
\begin{align}\label{eq:sup-Rayleigh}
  \alpha^+ := \sup_{\|w\|=1} \langle \mathbb{L}w,w\rangle
  = \|\mathbb{L}\|_{\op}. %
\end{align}
For $w\in \Hil$ with $\|w\|=1$, clearly $\langle \mathbb{L}w,w\rangle=1-\langle \mathbb{M}w,w\rangle$, so that $\lambda=1-\alpha^{+}$.

\emph{Step~2.} By the definition of $\alpha^+$, there exists a sequence $w_n=u_n\oplus v_n\in \Hil$ with
$\|w_n\| = 1$ and $\langle \mathbb{L}w_n,w_n\rangle \to \alpha^+$.
As $\|\mathbb{L}\|_{\op} = \alpha^+$ implies
$\|\mathbb{L}w_n\| \le \alpha^+$ for all~$n$, we have
\begin{align}\label{eq:approx-eig-L}
  \|\lambda w_{n}-\mathbb{M} w_n\|^2 
  = \|(\mathbb{L}-\alpha^+\Id)w_n\|^2 
  &= \|\mathbb{L}w_n\|^2 + (\alpha^+)^2 
     - 2\alpha^+\langle \mathbb{L}w_n,w_n\rangle \nonumber\\
  &\le 2\alpha^+ \big(\alpha^+ - \langle \mathbb{L}w_n,w_n\rangle\big)
   \to 0.
\end{align}
Below, we prove that $w_n$ converges strongly to a limit $w\neq0$ (along a subsequence). Once that is shown, we have $\lambda w=\mathbb{M} w$ and $\|w\|=1$, completing the proof.

\emph{Step~3.} We may choose representatives $(u_{n},v_{n})\in L^{2}(P)\times L^{2}(Q)$ of $w_n=u_n\oplus v_n$ with $\int u_n\,dP = \int v_n\,dQ$. 
Recall the operators $\mathbb{M}_{1}$ and $\mathbb{M}_{2}$ with $\mathbb{M}(u\oplus v)=\mathbb{M}_{1}(u,v)\oplus \mathbb{M}_{2}(u,v)$ from \cref{le:M-explicit}. Fubini's theorem shows that this split gives a balanced normalization,
\begin{align*}%
  \int \mathbb{M}_{1}(u,v)\, dP= \int \mathbb{M}_{2}(u,v)\,dQ.
\end{align*} 
Note also that if $(f_{n},g_{n})_{n\geq0}\subset L^{2}(P)\times L^{2}(Q)$ satisfy $\int f_{n}\,dP= \int g_{n}\,dQ$ for $n\geq0$, then the convergence $f_{n}\oplus g_{n}\to f_{0}\oplus g_{0}$ in $L^{2}(P\otimes Q)$ already implies the separate convergences $f_{n}\to f_{0}$ in $L^{2}(P)$ and $g_{n}\to g_{0}$ in $L^{2}(Q)$, thanks to the fact that
\[
f_n(\cdot)
= \int f_{n}(\cdot)\oplus g_{n}(y)\, dQ(y)
   - \frac{1}{2}\int f_{n}\oplus g_{n} \, d(P \otimes Q)
\]
and analogously for $g_n$. As a consequence, \eqref{eq:approx-eig-L} and $\mathbb{M}(u\oplus v)=\mathbb{M}_{1}(u,v)\oplus \mathbb{M}_{2}(u,v)$ imply
\begin{align}
  \big\|
    u_n(\lambda - Q(\cS_{r,(\cdot)}))
      - \mathbb{A}_1 v_n +\tfrac12 m(u_n\oplus v_n)
  \big\|_{L^2(P)}
  &\to 0, \label{eq:weakstrong1}\\
  \big\|
    v_n(\lambda - P(\cT_{r,(\cdot)}))
      - \mathbb{A}_2 u_n +\tfrac12 m(u_n\oplus v_n)
  \big\|_{L^2(Q)}
  &\to 0, \nonumber%
\end{align}
where we have abbreviated the integral terms with the operators
\begin{align*}
  \mathbb{A}_1 &: L^2(Q)\to L^2(P),
  \qquad
  \mathbb{A}_1 v(x) := \int_{\cS_{r,x}} v(y)\,dQ(y), \\
  \mathbb{A}_2 &: L^2(P)\to L^2(Q),
  \qquad
  \mathbb{A}_2 u(y) := \int_{\cT_{r,y}} u(x)\,dP(x). 
\end{align*}
These are integral operators with kernels in $L^2(P\otimes Q)$ and hence compact (cf.\ \cite[Theorem~6.12]{Brezis2011}).

As the sequences $(u_n)\subset L^2(P)$ and $(v_n)\subset L^2(Q)$ are bounded, the Banach–Alaoglu theorem yields (after passing to another subsequence) the weak convergences
\[
  u_n \rightharpoonup u \ \text{ in }L^2(P),\qquad
  v_n \rightharpoonup v \ \text{ in }L^2(Q),
\]
for some $u\in L^2(P)$ and $v\in L^2(Q)$. By the compactness of $\mathbb{A}_1$ and $\mathbb{A}_2$, this implies 
\[
  \mathbb{A}_1 v_n \to \mathbb{A}_1 v \ \text{ in }L^2(P),
  \qquad
  \mathbb{A}_2 u_n \to \mathbb{A}_2 u \ \text{ in }L^2(Q).
\]
Moreover, $m$ is clearly weakly continuous, so that $m(u_n\oplus v_n)\to m(u\oplus v)$. 
We can now revisit~\eqref{eq:weakstrong1}: as $\mathbb{A}_1 v_n \to \mathbb{A}_1 v$ strongly and $m(u_n\oplus v_n)\to m(u\oplus v)$ are constants, $u_n(\lambda - Q(\cS_{r,(\cdot)}))$ must also converge strongly. Multiplication with $1/(\lambda - Q(\cS_{r,(\cdot)}))$ is a bounded operator thanks to the uniform bound~\eqref{eq:greater-lambda}, so that $u_{n}$ also converges strongly; the limit must coincide with the weak limit~$u$. Similarly, $v_n \to v$. As noted after~\eqref{eq:approx-eig-L}, this completes the proof.
\end{proof}

Now we show the main result of this section. 

\begin{theorem}[Uniform coercivity of $\mathbb{M} $] \label{Theorem:Bound-double} Let $r_0$ be as in~\eqref{eq:def-C-r0} and $\delta_P$ as in~\cref{rk:coneCondition}. Set
$$\beta_\eps :=  \frac14 \kappa \alpha, \quad \text{where} \quad \kappa :=\delta_P\min\left( \frac{\eps}{8L},1\right)^d \leq 1,\quad 
\alpha := \frac{\lambda_P^2}{\Lambda_P^2}
\frac{\inf_{y\in\Omega'} Q\bigl(B_{\frac{\eps}{8L}}(y)\bigr)}{\bigl(\lceil 8L\,{\rm diam}(\Omega)/\eps \rceil\bigr)^{d+2}}  \leq 1.$$
Then for all $r\leq r_0$ and $(u,v)\in L^{2}(P)\times L^2(Q)$, the operator $\mathbb{M}=\mathbb{M}(r)$ of \eqref{eq:M} satisfies
\begin{equation}\label{eq:Mr-coercive}
\langle u\oplus v,  \mathbb{M} (u\oplus v)\rangle_{L^2(P\otimes Q)} \geq  \beta_\eps\|u\oplus v\|_{L^2(P\otimes Q)}^2, %
\end{equation}
and as a consequence, for a.e.~$r\leq r_0$,
    \begin{align}\label{eq:Bound-double-takeaway}
    \|(f_*-f) \oplus (g_*-g)\|_{L^2(P\otimes Q)}^2 \leq \frac{\eps}{\beta_\eps}\phi''(r).
    \end{align}
\end{theorem}

\begin{proof}%
\emph{Step~1.} We first derive lower bounds on the sections $\cT_{r,y}$ and $\cS_{r,x}$ of $\{f_r\oplus g_r\ge c\}$. Recall from \cref{rk:ballContained} and from the beginning of the proof of \cref{lemma:PL-forVariance} that 
 $$B_{\frac{\eps}{8L}}(x(y))\cap \Omega \;\subset\; \cT^\eps_{y}:= \left\{ x: f_*(x) + g_*(y)-c(x,y)\geq  \tfrac{\eps}{2}  \right\}\;\subset\;  \cT_{r,y}.$$
In view of \cref{rk:coneCondition}, it follows that
 \begin{equation}\label{eq:Measure-ball}
     P(\cT_{r,y})\geq \delta_P\min\left( \frac{\eps}{8L},1\right)^d = \kappa \quad \text{for all $y\in \Omega'$}.
\end{equation}    
Analogously, $\cS^\eps_{x}:= \{ y: f_*(x) + g_*(y)-c(x,y)\geq  \frac{\eps}{2}\}\;\subset\;  \cS_{r,x}$, so that \cref{eq:ballContained} yields
 \begin{equation}\label{eq:Measure-ball2}
     Q(\cS_{r,x})\geq \inf_{y\in\Omega'}Q\bigl(B_{\frac{\eps}{8L}}(y)\bigr)=:q_0 \quad \text{for all $x\in \Omega$}.
\end{equation}    

\emph{Step~2.} Let $\lambda_0$ be as in \cref{lemma:self-adj}. If $\lambda_0 \geq \beta_{\eps}$, then~\eqref{eq:Mr-coercive} holds and we are done. We may thus assume that $\lambda_0 \leq \beta_{\eps}$. Noting that $\kappa,\alpha\leq1$ and $\alpha\leq q_0$, the definition of $\beta_{\eps}$, \cref{eq:Measure-ball,eq:Measure-ball2} then yield
 \begin{align}
     P(\cT_{r,y}) - \lambda_0 &\geq \tfrac34 P(\cT_{r,y}) \geq \tfrac34 \kappa >0 \quad \text{for all $y\in \Omega'$,} \label{eq:lowerBound-kappa}\\
     Q(\cS_{r,x}) - \lambda_0 &\geq \tfrac34 Q(\cS_{r,x}) \geq \tfrac34 q_0 >0 \quad \text{for all $x\in \Omega$}. \nonumber
\end{align} 
In particular, we may assume that~\eqref{eq:greater-lambda} holds. 

\emph{Step~3.} Hence, by \cref{lemma:self-adj}, it suffices to show that the minimal eigenvalue $\lambda_0$ of $\mathbb{M}$ satisfies $\lambda_0\geq\beta_\eps$. Let $w=u\oplus v\neq0$ be an eigenvector of $\mathbb{M}$ with associated eigenvalue $\lambda_0\geq 0$ and suppose for contradiction that 
\begin{align}\label{eq:forContrad}
    \lambda_0<\beta_\eps.
\end{align}
We may choose representatives $(u,v)\in L^{2}(P)\times L^2(Q)$ with the centering $\int u \, dP=0$. The eigenvalue equation $\lambda_0 (u\oplus v) = \mathbb{M} (u\oplus v)$, the centering of~$u$ and~\cref{le:M-explicit} imply that
$$ \lambda_0 u = u Q(\cS_{r,(\cdot)}) + \int_{\cS_{r,(\cdot)}}   v\, dQ -  \int \left(u Q(\cS_{r,(\cdot)}) + \int_{\cS_{r,(\cdot)}}   v dQ\right) dP  $$
 and 
 \begin{equation}
     \label{eq:-for-v-eigenvalue} \lambda_0 v = v P(\cT_{r,(\cdot)}) + \int_{\cT_{r,(\cdot)}}   u \,dP.
 \end{equation}
On the other hand, the eigenvalue equation, \cref{lemma:PL-forVariance} and the centering of $u$ yield
 \begin{align*}
     \lambda_0 \| u\oplus v\|_{L^2(P\otimes Q)}^2  &=\langle u\oplus v,  \mathbb{M} (u\oplus v)\rangle_{L^2(P\otimes Q)}\geq \alpha
\,\Var_P(u) = \alpha {\|u\|_{L^2(P)}^2}.
 \end{align*}
 The centering also implies $\| u\oplus v\|_{L^2(P\otimes Q)}^2= \|u\|_{L^2(P)}^2+ \|v\|_{L^2(Q)}^2 $, and we conclude that 
 \begin{equation}
     \label{eq:lambda-larger-than-alpha}
\lambda_0 \geq \alpha \frac{\|u\|_{L^2(P)}^2}{\|u\|_{L^2(P)}^2+ \|v\|_{L^2(Q)}^2} .\end{equation}
Next, we bound $\|v\|_{L^2(Q)}^2$. Note that~\cref{eq:-for-v-eigenvalue} yields the inequality
$$  |v|\leq  \frac{\int_{\cT_{r,(\cdot)}}   |u| dP}{|\lambda_0-P(\cT_{r,(\cdot)})|} . $$
Using~\eqref{eq:lowerBound-kappa}, we deduce
$$  |v|\leq  \frac{\int_{\cT_{r,(\cdot)}}   |u| dP}{P(\cT_{r,(\cdot)})-\lambda_0} \leq \frac{\int_{\cT_{r,(\cdot)}}   |u| dP}{ \frac34 P(\cT_{r,(\cdot)})}  . $$
Taking squares, applying Jensen's inequality for $P(dx)/P(\cT_{r,y})$, integrating with respect to~$Q$, and using \cref{eq:Measure-ball} again, we get
$$ \|v\|_{L^2(Q)}^2 
\leq \frac{16}{9}\int \frac{\int_{\cT_{r,(\cdot)}}   u^2 dP}{P(\cT_{r,(\cdot)})} dQ 
\leq  \frac{16}{9}\frac{\|u\|_{L^2(P)}^2}{\kappa}\leq 2\frac{\|u\|_{L^2(P)}^2}{\kappa}.$$
As $u\oplus v \neq 0$, this shows in particular that $u\neq0$. 
Therefore, substituting the last display into~\cref{eq:lambda-larger-than-alpha} and using $\kappa\leq1$ yields
$$
  \lambda_0 \geq \alpha \frac{\|u\|_{L^2(P)}^2}{\|u\|_{L^2(P)}^2+2\|u\|_{L^2(P)}^2/\kappa}=\alpha \frac{\kappa}{\kappa+2} \geq\frac13 \alpha \kappa \geq \beta_\eps,
$$
contradicting~\eqref{eq:forContrad} and completing the proof of~\eqref{eq:Mr-coercive}.

\emph{Step~4.} Finally, set $w_* := (f_*-f)\oplus(g_*-g)$; then~\eqref{eq:phiSecondDeriv} and~\eqref{eq:Mr-coercive} yield 
\[
\phi''(r)
= \frac{1}{\eps}
\int_{\{f_r\oplus g_r\ge c\}}
w_* ^2\,d(P\otimes Q)
= \frac{1}{\eps}\,\langle w_*,\mathbb M w_*\rangle_{L^2(P\otimes Q)} \geq \frac{\beta_\eps}{\eps}\|w_*\|_{L^2(P\otimes Q)}^2,
\]
which is \eqref{eq:Bound-double-takeaway}.
\end{proof}

\subsection{Conclusion of the proof of \cref{th:PL-Phi}}

We can now complete the proof of the error bound and the PL inequality for $\Phi$.

\begin{proof}[Proof of \cref{th:PL-Phi}]
Using~\cref{eq:Bound-double-takeaway} in~\cref{development-Phi} yields
$$
\phi(r_0)\ge \phi(0)+ \frac{r_0^2\beta_\eps}{2\eps}{\|(f_*-f) \oplus (g_*-g)\|_{L^2(P\otimes Q)}^2} . $$
On the other hand, the convexity of  $\phi$ and the resulting monotonicity of $\phi' $ show
$$ \phi(r_0) \leq  \phi(0)+ r_0 \phi'(r_0)\leq  \phi(0)+ r_0 \phi'(1).$$
Together, we obtain
\begin{multline*}
    \frac{r_0^2 \beta_\eps}{2\eps}{\|(f_*-f) \oplus (g_*-g)\|_{L^2(P\otimes Q)}^2} 
    \leq \phi(r_0)- \phi(0)
    \leq r_0 \phi'(1) \\
    = r_0 \langle {\rm D}\Phi(f \oplus g),f_*\oplus g_*-f\oplus g \rangle_{L^2(P\otimes Q)}
    \leq r_0 \|{\rm D}\Phi(f \oplus g)\|_{L^2(P\otimes Q)} \|(f_*-f) \oplus (g_*-g)\|_{L^2(P\otimes Q)}
\end{multline*}
and hence
\begin{align}\label{eq:provedFirstClaim}
    \|(f_*-f) \oplus (g_*-g)\|_{L^2(P\otimes Q)} 
    \leq \frac{2\eps}{r_0 \beta_\eps} \|{\rm D}\Phi(f \oplus g)\|_{L^2(P\otimes Q)}.
\end{align}
Noting that 
$$
  \frac{2\eps}{r_0}=\frac{2\eps}{\min\bigl(\tfrac{\eps}{2C_{f,g}},1\bigr)}=\frac{4}{\min(1/  C_{f,g},2/\eps)}=4\max(C_{f,g},\eps/2)\leq 4\max(C_{f,g},\eps)
$$
and
\begin{align*}
   \beta_\eps^{-1} 
   = 4 \kappa^{-1} \alpha^{-1}
   = 4 \left(\delta_P^{-1}\max\left( \frac{8L}{\eps},1\right)^d\right) \frac{\Lambda_P^2}{\lambda_P^2}
\frac{\bigl(\lceil 8L\,{\rm diam}(\Omega)/\eps \rceil\bigr)^{d+2}}{\inf_{y\in\Omega'} Q\bigl(B_{\frac{\eps}{8L}}(y)\bigr)},
\end{align*}
this gives the first claim of \cref{th:PL-Phi}. Moreover, concavity of $\Phi$ yields
\begin{align*}
\QOT_\eps(P,Q) -\Phi(f \oplus g)=\Phi(f_* \oplus g_*) -\Phi(f \oplus g) \leq \langle {\rm D}\Phi(f \oplus g),f_*\oplus g_*-f\oplus g \rangle_{L^2(P\otimes Q)}\\
\leq \|{\rm D}\Phi(f \oplus g)\|_{L^2(P\otimes Q)} \|(f_*-f) \oplus (g_*-g)\|_{L^2(P\otimes Q)}
\end{align*}
and now bounding $\|(f_*-f) \oplus (g_*-g)\|_{L^2(P\otimes Q)}$ by~\eqref{eq:provedFirstClaim} yields the second claim.
\end{proof}

\subsection{Proof of \cref{th:PL-Gamma}}

Finally, we translate the result of \cref{th:PL-Phi} from~$\Phi$ to~$\Gamma$, that is, from the space~$\Hil$ to $L^{2}(P) \times L^{2}(Q)$.

\begin{lemma}\label{le:translationPhiToGamma}
    Let $(f,g) \in L^{2}(P) \times L^{2}(Q)$ and 
    \begin{align*}
        I_{0} := \int\left( 1 - \frac{1}{\eps}(f \oplus g - c)_{+} \right) d(P\otimes Q).
    \end{align*}
    The gradients of $\Phi: \Hil \to\R$ and $\Gamma: L^{2}(P) \times L^{2}(Q)\to\R$ are related by
    \begin{align}\label{eq:gradientsRelation}
        {\rm D}\Phi(f \oplus g) = {\rm D}_{1}\Phi(f \oplus g) \oplus {\rm D}_{2}\Phi(f \oplus g), \quad\text{where}\quad 
        {\rm D}_{i}\Phi(f \oplus g) = {\rm D}_{i}\Gamma(f,g) - \frac{1}{2} I_{0}
    \end{align}
    and the explicit form of ${\rm D}_{i}\Gamma(f,g)$ is given in~\eqref{eq:gradientGamma}.
    As a consequence, 
    \begin{align*}
        \|{\rm D}\Phi(f \oplus g)\|^{2}_{L^{2}(P \otimes Q)} 
        = \|{\rm D}\Gamma(f, g)\|^{2}_{L^{2}(P) \times L^{2}(Q)} - I^{2}_{0} \leq \|{\rm D}\Gamma(f, g)\|^{2}_{L^{2}(P) \times L^{2}(Q)}.
    \end{align*}
\end{lemma}

\begin{proof}
    For all $(f,g),(u,v)\in L^{2}(P)\times L^2(Q)$, we have
    \begin{align*}
        \langle {\rm D}\Phi(f \oplus g), u\oplus v \rangle_{L^{2}(P\otimes Q)} 
        &=  \int (u \oplus v)\left( 1 - \frac{1}{\eps}(f \oplus g - c)_{+} \right) d(P \otimes Q) \\
        & = \big\langle 1 - \tfrac{1}{\eps}(f \oplus g - c)_{+}, u\oplus v \big\rangle_{L^{2}(P\otimes Q)}\\
        & = \langle {\rm D} \Gamma(f,g), (u,v)\rangle_{L^{2}(P)\times L^2(Q)}.
    \end{align*}
    Now~\eqref{eq:gradientsRelation} follows as in the proof of \cref{le:M-explicit}. Moreover, using the above identity with $u={\rm D}_{1}\Phi(f \oplus g)$ and $v={\rm D}_{2}\Phi(f \oplus g)$ yields
    \begin{align*}
        \|{\rm D}\Phi(f \oplus g)\|^{2}_{L^{2}(P \otimes Q)} 
        &=  \| {\rm D}_{1}\Gamma(f,g) \|^{2}_{L^{2}(P)} +  \| {\rm D}_{2}\Gamma(f,g) \|^{2}_{L^{2}(Q)} \\
        &\,\quad - \frac12 \langle {\rm D}_{1}\Gamma(f,g), I_{0}\rangle_{L^{2}(P)} - \frac12 \langle {\rm D}_{2}\Gamma(f,g), I_{0}\rangle_{L^{2}(Q)} \\
        & =   \| {\rm D}\Gamma(f,g) \|^{2}_{L^{2}(P)\times L^2(Q)} - I_0^2. \qedhere
    \end{align*}
\end{proof}

\begin{proof}[Proof of \cref{th:PL-Gamma}]
    Combine \cref{th:PL-Phi} with $\Gamma(f,g)=\Phi(f \oplus g)$ and \cref{le:translationPhiToGamma}.
\end{proof}

\section{Proofs of linear convergence} \label{se:proofs-of-linear}

We first record several Lipschitz constants associated with ${\rm D}\Gamma$. Recall the components ${\rm D}_{1}\Gamma(f,g)\in L^2(P)$ and ${\rm D}_{2}\Gamma(f,g)\in L^2(Q)$ from~\eqref{eq:gradientGamma}.

\begin{lemma}\label{le:lipschitzgradient}
    The gradient ${\rm D}\Gamma: L^{2}(P) \times L^{2}(Q)\to L^{2}(P) \times L^{2}(Q)$ is $\frac{2}{\eps}$-Lipschitz. Moreover, ${\rm D}_{i}\Gamma(f,\cdot)$ and ${\rm D}_{i}\Gamma(\cdot,g)$ are $\frac{1}{\eps}$-Lipschitz for $i=1,2$, for all $(f,g)\in L^{2}(P) \times L^{2}(Q)$.
\end{lemma}

\begin{proof}
    For $(f,g),(u,v) \in L^{2}(P) \times L^{2}(Q)$, the inequality $|(t)_{+} - (s)_{+}| \leq |t-s|$ yields
    \begin{align*}
        \big|{\rm D}_{1}\Gamma(f,g) - {\rm D}_{1}\Gamma(u,v) \big|
        &\leq \frac{1}{\eps} \int \left|\big(u(\cdot) + v(y) - c(\cdot,y)\big)_{+} - \big(f(\cdot) + g(y) - c(\cdot, y)\big)_{+}\right| dQ(y) \\
        & \leq \frac{1}{\eps} \left\{ |f(\cdot) - u(\cdot)| + \int |g(y) - v(y)| dQ(y) \right\}.
    \end{align*}
    Using $(a+b)^2\leq 2a^2+2b^2$ and Jensen's inequality, we deduce
    \begin{align}
        \| {\rm D}_{1}\Gamma(f,g) - {\rm D}_{1}\Gamma(u,v) \|^{2}_{L^{2}(P)} & \leq \frac{2}{\eps^{2}} \left( \|f - u\|_{L^{2}(P)}^{2} + \|g - v\|_{L^{2}(Q)}^{2} \right) \nonumber\\
        & = \frac{2}{\eps^{2}} \| (f,g) - (u,v) \|^{2}_{L^{2}(P) \times L^{2}(Q)} \label{eq:D1Gamma-LipschitzEst}.
    \end{align}
    Note that if $f=u$ or $g=v$, the factor 2 is unnecessary, showing the claimed $\frac{1}{\eps}$-Lipschitz property of ${\rm D}_{1}\Gamma(f,\cdot)$ and ${\rm D}_{1}\Gamma(\cdot,g)$. On the other hand, combining~\eqref{eq:D1Gamma-LipschitzEst} with its analogue for ${\rm D}_{2}\Gamma$ yields
    \begin{align*}
        \| {\rm D}\Gamma(f,g) - {\rm D}\Gamma(u,v) \|^{2}_{L^{2}(P) \times L^{2}(Q)} \leq \frac{4}{\eps^{2}}\| (f,g) - (u,v) \|^{2}_{L^{2}(P) \times L^{2}(Q)}. \mbox{\qedhere}
    \end{align*}
\end{proof}

\begin{remark}\label{rk:sumNorm}
    For any bounded functions $f:\Omega\to\R$ and $g:\Omega'\to\R$, we note that 
\begin{align*}%
    \|f\oplus g\|_\infty = \inf_{a\in\R} (\|f-a\|_\infty + \|g+a\|_\infty).
\end{align*}
Recall that if $(f_*,g_*)$ are potentials, then $(f_*-a,g_*+a)$ are also potentials, for any $a\in\R$. Thus, if some inequality
$$
  A(f,g)\leq B(f,g,\|f-f_*\|_\infty + \|g-g_*\|_\infty)
$$
has been established for arbitrary potentials $(f_*,g_*)$, we can already conclude that 
$$
  A(f,g)\leq B(f,g,\|f\oplus g - f_*\oplus g_*\|_\infty).
$$
\end{remark}

\subsection{Proof of \cref{coro:GD} (gradient ascent)}

This proof is standard; we state the argument for the sake of completeness.

\begin{proof}[Proof of \cref{coro:GD}]
The bound \eqref{eq:GDIterateBound} for the iterates is shown in \cite[Lemma~5.2]{GonzalezSanzNutzRiveros.25}. (The norm in that reference is defined differently but has the same value by \cref{rk:sumNorm}. Moreover, the reference assumes that~$c$ is the quadratic cost, but the proof applies to any bounded~$c$.)

By \cref{le:lipschitzgradient}, ${\rm D}\Gamma$ is $\frac{2}{\eps}$-Lipschitz on $L^{2}(P)\times L^{2}(Q)$, hence $\Gamma$ is $\frac{2}{\eps}$-smooth. In particular,
\begin{equation}\label{eq:smoothness-Gamma}
\Gamma(f+u,g+v)\ge \Gamma(f,g) + \big\langle {\rm D}\Gamma(f,g),(u,v)\big\rangle_{L^{2}(P)\times L^{2}(Q)}
-\frac{1}{\eps}\|(u,v)\|_{L^{2}(P)\times L^{2}(Q)}^{2}.
\end{equation}
We apply~\eqref{eq:smoothness-Gamma} with $(f,g)=(f_n,g_n)$ and $(u,v)=\eta\,  {\rm D}\Gamma(f_n,g_n)$. Using the update rule~\eqref{eq:gradientascent} and $\eta\in(0,\eps)$, we obtain
\begin{align*}
\Gamma(f_{n+1},g_{n+1})
&\ge \Gamma(f_n,g_n)
+\eta\|{\rm D}\Gamma(f_n,g_n)\|^2_{L^{2}(P)\times L^{2}(Q)}
-\frac{\eta^{2}}{\eps}\|{\rm D}\Gamma(f_n,g_n)\|_{L^{2}(P)\times L^{2}(Q)}^{2} \\
&=\Gamma(f_n,g_n)+\eta\Bigl(1-\frac{\eta}{\eps}\Bigr)\|{\rm D}\Gamma(f_n,g_n)\|^2_{L^{2}(P)\times L^{2}(Q)}.
\end{align*}
Subtracting this from the optimal value $\QOT_\eps(P,Q)$ gives
\begin{equation}\label{eq:GD-gap-step}
\Delta_{n+1}
\le \Delta_n-\eta\Bigl(1-\frac{\eta}{\eps}\Bigr)\|{\rm D}\Gamma(f_n,g_n)\|^2_{L^{2}(P)\times L^{2}(Q)}.
\end{equation}
Using the PL inequality~\eqref{eq:PLforGamma} at $(f_n,g_n)$ and the bound~\eqref{eq:GDIterateBound},
\begin{equation*}
\|{\rm D}\Gamma(f_n,g_n)\|_{L^{2}(P)\times L^{2}(Q)}^{2}
\;\ge\; \frac{1}{\gamma_\eps M}\,\Delta_n, \qquad M:= \max(2\|f_{0}\oplus g_{0} - f_*\oplus g_*\|_\infty,\eps),
\end{equation*}
hence we conclude that 
\[
\Delta_{n+1}
\le \Delta_n-\eta\Bigl(1-\frac{\eta}{\eps}\Bigr)\frac{1}{\gamma_\eps M}\Delta_n
= (1-q)\Delta_n,
\qquad
q:=\frac{1}{\gamma_{\eps}M}\,\eta\Bigl(1-\frac{\eta}{\eps}\Bigr). %
\]

Iterating this estimate yields $\Delta_n\le (1-q)^n\Delta_0$.

It remains to prove the bound for the iterates in $L^2(P\otimes Q)$. By the error bound in \cref{th:PL-Gamma} and the definition of $M$,
\begin{equation}\label{eq:GD-error-bound-step}
\|f_n\oplus g_n - f_*\oplus g_*\|_{L^2(P\otimes Q)}^2
\le \gamma_\eps^2 M^2\|{\rm D}\Gamma(f_n,g_n)\|_{L^{2}(P)\times L^{2}(Q)}^2.
\end{equation}
On the other hand, \eqref{eq:GD-gap-step} implies
\[
\eta\Bigl(1-\frac{\eta}{\eps}\Bigr)\|{\rm D}\Gamma(f_n,g_n)\|_{L^{2}(P)\times L^{2}(Q)}^{2}
\le \Delta_n-\Delta_{n+1}
\le \Delta_n
\le (1-q)^n\Delta_0.
\]
Substituting this estimate into \eqref{eq:GD-error-bound-step}, we conclude that
\[
\|f_n\oplus g_n - f_*\oplus g_*\|_{L^2(P\otimes Q)}^2
\le \frac{\gamma_\eps^2 M^2}{\eta(1-\eta/\eps)}\,\Delta_0\,(1-q)^n. \qedhere
\]

\end{proof}

\subsection{Proof of \cref{coro:Implicit} (coordinate ascent)}

Again, the proof has two steps. The first step is to bound the iterates. Once that is achieved, the second step follows the standard argument for deriving linear convergence of coordinate ascent from a PL inequality. 

\begin{proof}[Proof of \cref{coro:Implicit}] 
\emph{Step 1.} Set
\begin{align}\label{eq:sectionsForLinearConvProofs}
    \cS_{x} := \left\{ y \in \Omega' \text{ : } f_{*}(x) + g_{*}(y) \geq c(x,y) \right\}, \quad \cS^{n}_{x} := \left\{ y \in \Omega' \text{ : } f_{n}(x) + g_{n}(y) \geq c(x,y) \right\}.
\end{align}
Using \cref{eq:FOC} and the definition~\eqref{eq:implicitalgo-f} of  $f_{n}$, as well as the inequality $(a)_{+} - (b)_{+} \leq (a - b)\1_{a \geq 0}$, we have for every $x \in \Omega$ that
\begin{align*}
    0 & = \int \big(f_{*}(x) + g_{*}(y) - c(x,y)\big)_{+} - \big(f_n(x) + g_{n}(y) - c(x,y)\big)_{+} dQ(y) \\
    & \leq \int_{\cS_{x}} f_{*}(x) + g_{*}(y) - f_n(x) - g_{n}(y) dQ(y)
\end{align*}
and hence
\begin{align*}
    (f_n(x) - f_{*}(x)) Q(\cS_{x}) \leq  \int_{\cS_{x}} g_{*}(y) -  g_{n}(y) dQ(y) \leq \| g_{n} - g_{*} \|_{\infty} Q(\cS_{x}).
\end{align*}
Symmetrically, 
$
    (f_{*}(x) - f_n(x)) Q(\cS^{n}_{x}) \leq \| g_{n} - g_{*} \|_{\infty} Q(\cS^{n}_{x}).
$
Noting that \cref{eq:FOC} and \cref{eq:implicitalgo-f} imply  $Q(\cS_{x}) > 0$ and $Q(\cS^{n}_{x}) > 0$, we conclude
\begin{align*}
    \| f_n - f_{*} \|_{\infty} \leq \| g_{n}  - g_{*} \|_{\infty}.
\end{align*}
Arguing analogously for $g_{n+1}$ yields 
\begin{align*}
   \| g_{n+1}  - g_{*} \|_{\infty} \leq  \| f_n - f_{*} \|_{\infty}.
\end{align*}
Combining these two inequalities yields the claim~\eqref{eq:implicitIteratesBound}.

\emph{Step 2.}
Recall the expressions of ${\rm D}_1 \Gamma $ and ${\rm D}_2\Gamma $ from \eqref{eq:gradientGamma}. By \cref{le:lipschitzgradient},  ${\rm D}_i\Gamma$ is $\frac{1}{\eps}$-Lipschitz in each of its coordinates. In particular, the descent lemma yields
$$ \Gamma(f,g_n)\geq \Gamma(f_{n-1},g_n) + \langle   {\rm D}_1 \Gamma(f_{n-1} ,g_{n}) , f-f_{n-1} \rangle_{L^2(P)} - \frac{1}{2\eps} \| f-f_{n-1}\|^2_{L^2(P)}$$
for every $f\in L^2(P)$. Taking supremum over $f$ and recalling the definition of~$f_n$, we obtain
\begin{align*}
 \Gamma(f_n,g_n)&\geq \Gamma(f_{n-1},g_n) + \sup_f \left\{\langle   {\rm D}_1 \Gamma(f_{n-1} ,g_{n}) , f-f_{n-1} \rangle_{L^2(P)}  -  \frac{1}{2\eps} \| f-f_{n-1}\|^2_{L^2(P)} \right\}\\
& = \Gamma(f_{n-1},g_n) +\frac{\eps}{2} \| {\rm D} _1\Gamma(f_{n-1} ,g_{n}) \|^2_{L^2(P)}.  \end{align*}
Analogously,
$$  \Gamma(f_{n-1},g_n)\geq \Gamma(f_{n-1},g_{n-1}) +\frac{\eps}{2} \| {\rm D}_2\Gamma(f_{n-1} ,g_{n-1}) \|^2_{L^2(Q)}.$$
Combining the two inequalities yields
\begin{align*}
 \Gamma(f_n,g_n)\geq \Gamma(f_{n-1},g_{n-1})  +\frac{\eps}{2}\left( \| {\rm D} _1\Gamma(f_{n-1} ,g_{n}) \|^2_{L^2(P)}  +\| {\rm D}_2\Gamma(f_{n-1} ,g_{n-1}) \|^2_{L^2(Q)} \right),\end{align*}
 but as ${\rm D}_1\Gamma(f_{n-1} ,g_{n-1}) =0$ and ${\rm D}_2\Gamma(f_{n-1} ,g_{n}) = 0 $ by the first-order conditions~\cref{eq:implicitalgo-f,eq:implicitalgo-g} for $f_{n-1}$ and $g_n$, respectively, this even gives
\begin{align}
    \Gamma(f_n,g_n)&\geq \Gamma(f_{n-1},g_{n-1})  +\frac{\eps}{2}\left( \| {\rm D} \Gamma(f_{n-1} ,g_{n}) \|^2_{L^2(P) \times L^{2}(Q)} +\| {\rm D}\Gamma(f_{n-1} ,g_{n-1}) \|^2_{L^2(P) \times L^{2}(Q)}\right) \nonumber\\
    &\geq \Gamma(f_{n-1},g_{n-1})  +\frac{\eps}{2} \| {\rm D}\Gamma(f_{n-1} ,g_{n-1}) \|^2_{L^2(P) \times L^{2}(Q)} \label{eq:beforePLapplication}.
\end{align}
Next, we apply the PL inequality \cref{eq:PLforGamma} to the last term. Using that
$$
  \|f_{n-1}\oplus g_{n-1} - f_*\oplus g_*\|_\infty \leq \|f_{n-1}-f_*\|_{\infty}+\|g_{n-1}-g_*\|_{\infty} \leq 2 \|g_{0}-g_*\|_{\infty}
$$
by~\eqref{eq:implicitIteratesBound} and abbreviating $M:=\max\left(2\|g_{0}-g_*\|_{\infty},\eps\right)$, we obtain
\begin{align*}
    \Gamma(f_n,g_n)\geq \Gamma(f_{n-1},g_{n-1}) +\frac{\eps}{2}\frac{1}{\gamma_\eps  M} \left(  \QOT_\eps(P,Q)- \Gamma(f_{n-1},g_{n-1})\right).
\end{align*}
Subtracting this from $\QOT_\eps(P,Q)$ gives
 $   \Delta_n
    \leq  \Delta_{n-1}  (1-\frac{\eps}{2\gamma_\eps  M})$ and hence the claim for the suboptimality gap. The claim for the iterates follows via the error bound, as in the proof of \cref{coro:GD}.
\end{proof}

\subsection{Proof of \cref{coro:Explicit} (coordinate gradient ascent)}

Once again, the argument has two steps.

\begin{proof}[Proof of \cref{coro:Explicit}] 
\emph{Step 1.} Recall the notation~\eqref{eq:sectionsForLinearConvProofs}. The bound~\eqref{eq:ExplicitIterateBound} follows by a similar argument as, e.g., \cite[Lemma~5.2]{GonzalezSanzNutzRiveros.25}; we give the details for completeness. From \cref{Explicit-coordinate} we obtain that for all $x \in \Omega$,
\begin{align*}
    f_{n+1}(x)  - f_{*}(x) = & f_{n}(x) - f_{*}(x) + \frac{\eta}{\eps} \left( \eps - \int (f_{n}(x) + g_{n}(y) - c(x,y))_{+} dQ(y) \right) \\
    = & f_{n}(x) - f_{*}(x) \\
    & + \frac{\eta}{\eps} \left( \int (f_{*}(x) + g_{*}(y) - c(x,y))_{+} - (f_{n}(x) + g_{n}(y) - c(x,y))_{+} dQ(y) \right) \\
    & \leq f_{n}(x) - f_{*}(x) + \frac{\eta}{\eps} \left( \int_{\cS_{x}} f_{*}(x) + g_{*}(y) - (f_{n}(x) + g_{n}(y)) dQ(y) \right) \\
    & = \left( 1 - \frac{\eta}{\eps}Q(\cS_{x})\right) (f_{n}(x) - f_{*}(x)) - \frac{\eta}{\eps} \left( \int_{\cS_{x}}g_{n}(y) - g_{*}(y) dQ(y) \right) \\
    & \leq \left( 1 - \frac{\eta}{\eps}Q(\cS_{x})\right) \|f_{n} - f_{*}\|_{\infty} + \frac{\eta}{\eps} Q(\cS_{x}) \|g_{n} - g_{*}\|_{\infty},
\end{align*}
which is a convex combination of $\|f_{n} - f_{*}\|_{\infty}$ and $\|g_{n} - g_{*}\|_{\infty}$. Analogously,
\begin{align*}
    f_{*}(x)  - f_{n+1}(x) \leq \left( 1 - \frac{\eta}{\eps}Q(\cS^{n}_{x})\right) \|f_{n} - f_{*}\|_{\infty} + \frac{\eta}{\eps} Q(\cS^{n}_{x}) \|g_{n} - g_{*}\|_{\infty}.
\end{align*}
Together, we have
\begin{align*}
    \| f_{n+1} - f_{*} \|_{\infty} \leq \max ( \|f_{n} - f_{*}\|_{\infty}, \|g_{n} - g_{*}\|_{\infty}).
\end{align*}
Repeating the argument for $g_{n+1}$ gives
\begin{align*}
    \| g_{n+1} - g_{*} \|_{\infty} \leq  \max ( \|f_{n+1} - f_{*}\|_{\infty}, \|g_{n} - g_{*}\|_{\infty}) \leq  \max ( \|f_{n} - f_{*}\|_{\infty}, \|g_{n} - g_{*}\|_{\infty}) 
\end{align*}
and combining the two bounds yields~\eqref{eq:ExplicitIterateBound}.

\emph{Step~2.} 
\Cref{le:lipschitzgradient} implies
$$ \Gamma(f_{n+1},g_{n+1}) \geq  \Gamma(f_{n+1},g_{n}) + \left\langle {\rm D}_2\Gamma(f_{n+1},g_{n}),  g_{n+1}-g_n \right\rangle_{L^{2}(Q)} -\frac{1}{2\eps}\| g_{n+1}-g_n\|^{2}_{L^{2}(Q)}.$$
The definition of $g_{n+1}$ shows that $g_{n+1}-g_n= \eta\, {\rm D}_2\Gamma(f_{n+1},g_{n})$. Substituting this gives
\begin{equation}\label{eq:inequality-explicit-1}
    \Gamma(f_{n+1},g_{n+1}) \geq  \Gamma(f_{n+1},g_{n}) + \eta\left(1- \frac{\eta}{2\eps} \right) \left\| {\rm D}_2\Gamma(f_{n+1},g_{n})\right\|_{L^{2}(Q)}^2. 
\end{equation}
Analogously, we have
\begin{equation}
    \label{eq:inequality-explicit-2}
    \Gamma(f_{n+1},g_n) \geq  \Gamma(f_{n},g_{n}) + \eta\left(1 - \frac{\eta}{2 \eps } \right) \left\| {\rm D}_1\Gamma(f_{n}, g_{n})\right\|_{L^{2}(P)}^{2}. 
\end{equation}
Plugging \cref{eq:inequality-explicit-2} into \cref{eq:inequality-explicit-1} yields
\begin{equation}\label{eq:inequality-explicit-3}
    \Gamma(f_{n+1},g_{n+1}) \geq  \Gamma(f_{n},g_{n}) + \eta\left(1 - \frac{\eta}{2\eps} \right) \left( \left\| {\rm D}_1\Gamma(f_{n},g_{n})\right\|_{L^{2}(P)}^2 + \left\| {\rm D}_2\Gamma(f_{n+1}, g_{n})\right\|_{L^{2}(Q)}^{2} \right) . 
\end{equation}

Note that the inequality $\|u\|^2 \ge \frac12\|v\|^2 - \|u-v\|^2$ holds in any normed space. Applying this with $u := {\rm D}_2\Gamma(f_{n+1},g_n)\in L^2(Q)$ and $v := {\rm D}_2\Gamma(f_n,g_n)\in L^2(Q)$ yields
\begin{equation*}
\|{\rm D}_2\Gamma(f_{n+1},g_n)\|_{L^2(Q)}^2
\ge
\frac12 \|{\rm D}_2\Gamma(f_n,g_n)\|_{L^2(Q)}^2
-
\|{\rm D}_2\Gamma(f_{n+1},g_n)-{\rm D}_2\Gamma(f_n,g_n)\|_{L^2(Q)}^2.
\end{equation*}
By \cref{le:lipschitzgradient} and the definition of $f_{n+1}$, the last term satisfies
\[
\|{\rm D}_2\Gamma(f_{n+1},g_n)-{\rm D}_2\Gamma(f_n,g_n)\|_{L^2(Q)}
\le \frac{1}{\eps}\|f_{n+1}-f_n\|_{L^2(P)} 
= \frac{\eta}{\eps}\|{\rm D}_1\Gamma(f_{n},g_{n})\|_{L^2(P)},
\]
so that 
\begin{align*}
    \left\| {\rm D}_2\Gamma(f_{n+1},g_{n})\right\|_{L^{2}(Q)}^2 
&\geq  \frac{1}{2}\left\| {\rm D}_2\Gamma(f_{n},g_{n})\right\|_{L^{2}(Q)}^2 - \frac{\eta^{2}}{\eps^{2}} \| {\rm D}_1\Gamma(f_{n},g_{n}) \|^{2}_{L^{2}(P)}.
\end{align*}
Plugging this into \eqref{eq:inequality-explicit-3} and recalling that $ \eta\in (0, \eps/\sqrt{2})$, we deduce
\begin{align*}
    \Gamma(f_{n+1},g_{n+1}) &\geq  \Gamma(f_{n},g_{n}) + \frac{\eta}{2}\left(1 - \frac{\eta}{2\eps} \right) \left\| {\rm D}\Gamma(f_{n}, g_{n})\right\|_{L^{2}(P) \times L^{2}(Q)}^{2}.
\end{align*}
This bound has the same form as \eqref{eq:beforePLapplication} except for the different constant. Moreover, \eqref{eq:ExplicitIterateBound} implies that 
$\|f_{n} - f_{*}\|_{\infty}+ \|g_{n} - g_{*}\|_{\infty} \leq 2(\|f_{0} - f_{*}\|_{\infty}+ \|g_{0} - g_{*}\|_{\infty})$. We can now proceed exactly as after \eqref{eq:beforePLapplication} to deduce the final result, except that we also use \cref{rk:sumNorm} to replace $\|f_{0} - f_{*}\|_{\infty}+ \|g_{0} - g_{*}\|_{\infty}$ by $\|f_{0} \oplus g_{0} - f_{*}\oplus g_{*}\|_{\infty}$. Once again, the claim for the iterates follows via the error bound, as in the proof of \cref{coro:GD}.
\end{proof}

\section{Conclusion}\label{sec:conclusion}

In this paper we established a local Polyak--{\L}ojasiewicz inequality for the dual objective of quadratically regularized optimal transport. The proof proceeds by a uniform coercivity estimate for the second-order form obtained from the dual objective along the segment joining an arbitrary pair $(f,g)$ to an optimizer $(f_*,g_*)$. This yields the error bound and PL inequality of \cref{th:PL-Gamma}; the dependence of the constant on the $L^\infty$ distance to the optimizer reflects the local nature of the result which is unavoidable given the shape of the dual objective.

As an application, we derived linear convergence rates for three natural dual algorithms: gradient ascent \eqref{eq:gradientascent}, coordinate ascent \eqref{eq:implicit}, and coordinate gradient ascent \eqref{Explicit-coordinate}. Coordinate ascent is implicit for QOT, since each coordinate update requires solving \eqref{eq:implicitalgo-f}--\eqref{eq:implicitalgo-g}. By contrast, gradient ascent and coordinate gradient ascent are explicit schemes, but they require a choice of step size. The sufficient step-size assumptions in \cref{coro:GD,coro:Explicit} ensure monotonicity, uniform $L^\infty$ control of the iterates, and hence a uniform PL constant along the trajectory.

Qualitatively, the theoretical results predict the behavior observed numerically. In practice (see, e.g., the plots in \cite{GonzalezSanzNutzRiveros.25} for gradient ascent), we observe a finite burn-in phase followed by a clear linear convergence regime. The constants in \cref{Section:linear-convergence} are necessarily conservative: in practice the burn-in is typically much shorter, and the asymptotic contraction is often much faster than what is guaranteed by the explicit bounds. While the constants in \cref{th:PL-Gamma} have a pessimistic explicit dependence on $\eps$, numerical experiments often show that the number of iterations required to reach a fixed accuracy scales closer to $1/\eps$ over moderate ranges.

The step-size restrictions are likewise conservative. Our theory requires $\eta<\eps$ for gradient ascent and $\eta<\eps/\sqrt{2}$ for coordinate gradient ascent. These assumptions are sufficient for the estimates in \cref{eq:GDIterateBound,eq:ExplicitIterateBound}, but they are not sharp as stability thresholds. Numerically, convergence does fail when the step size is taken too large, but the breakdown often occurs at values larger than those imposed by the proof (see also \cite{GonzalezSanzNutzRiveros.25}). In particular, the  optimal step size (obtained a posteriori by grid search) is often relatively large.

Lastly, let us briefly comment on some practical aspects that are not captured by our theory. The comparison between the two explicit schemes is nuanced. Coordinate gradient ascent typically requires fewer outer iterations than gradient ascent, because each iteration uses the updated first coordinate before updating the second. However, each such iteration is more expensive. At the conservative step sizes guaranteed by \cref{coro:GD,coro:Explicit}, the wall-time performance of the two schemes is therefore often comparable. If, instead, the best stable step size is supplied by an oracle, coordinate gradient ascent is frequently faster; its advantage then stems from its ability to use a larger effective step size while retaining a faster contraction per iteration. In practice, an adaptive choice of step size (such as monotone Armijo backtracking with Barzilai--Borwein initialization) substantially accelerates both schemes. Then, we often observe that coordinate gradient ascent converges faster (certainly in iterations, but even when measured in wall-time), especially when $\eps$ is small. Gradient ascent is sometimes faster for large $\eps$, thanks to the simpler iteration. In both cases, the final step size is often a multiple of~$\eps$, and these large steps seem to be responsible for the improvement. The implicit coordinate ascent scheme does not require a step size, but instead requires solving the first-order equations. Implemented with an efficient solver (such as an active-set Newton method), coordinate ascent is competitive with the explicit methods, and often the fastest when $\eps$ is moderate---while the implicit nature makes each iteration costly, a better contraction constant more than compensates for that. Broadly speaking, the three methods seem comparable, with wall-time ratios rarely exceeding 3 for moderate experiments. However, extensive experiments and sophisticated implementations are left for future work.

\appendix

\section{Extension to uniformly continuous costs}\label{app:uniformly-continuous-costs}

This appendix shows how to generalize our results from Lipschitz costs $c$ to costs that are merely uniformly continuous. Set
\[
    R:=1\vee \diam(\Omega)\vee \diam(\Omega').
\]

\begin{assumption}[Cost with modulus of continuity]\label{assumption:cost-omega}
There exists a modulus of continuity $\omega:[0,\infty)\to[0,\infty)$, that is, a nondecreasing continuous function with $\omega(0)=0$, such that
\begin{equation}\label{eq:cost-omega}
    |c(x,y)-c(x',y')|\le \omega(\|x-x'\|+\|y-y'\|)
    \qquad\text{for all }x,x'\in\Omega,\ y,y'\in\Omega'.
\end{equation}
\end{assumption}

Under Assumption~\ref{assumption:cost-omega}, Proposition~\ref{pr:prelims} remains valid, except that item~(iii) is replaced by
\begin{equation}\label{eq:potentials-omega}
    |f_*(x)-f_*(x')|\le \omega(\|x-x'\|),
    \qquad
    |g_*(y)-g_*(y')|\le \omega(\|y-y'\|);
\end{equation}
cf.~\cite{Nutz.2024}. We fix these versions of the potentials throughout the remainder of the appendix.

To state the results, define the radius
\begin{equation}\label{eq:def-varrho-omega}
    \varrho_{\eps,\omega}
    :=
    \sup\Bigl\{r\in[0,R]: 2\omega(r)+\omega(2r)\le \frac{\eps}{2}\Bigr\}>0
\end{equation}
and the constants
\begin{align}
    \alpha_{\eps,\omega}
    &:= \frac{\lambda_P^2}{\Lambda_P^2}
    \frac{\inf_{y\in\Omega'}Q(B_{\varrho_{\eps,\omega}}(y))}{\bigl(\lceil \diam(\Omega)/\varrho_{\eps,\omega}\rceil\bigr)^{d+2}},
    \label{eq:def-alpha-omega}\\
    \kappa_{\eps,\omega}
    &:=\delta_P\min(\varrho_{\eps,\omega},1)^d,
    \qquad
    \beta_{\eps,\omega}:=\frac14\kappa_{\eps,\omega}\alpha_{\eps,\omega},
    \label{eq:def-kappa-beta-omega}\\
    \gamma_{\eps,\omega}
    &:=4\beta_{\eps,\omega}^{-1}
    =16\,\delta_P^{-1}\min(\varrho_{\eps,\omega},1)^{-d}
    \frac{\Lambda_P^2}{\lambda_P^2}
    \frac{\bigl(\lceil \diam(\Omega)/\varrho_{\eps,\omega}\rceil\bigr)^{d+2}}{\inf_{y\in\Omega'}Q(B_{\varrho_{\eps,\omega}}(y))}.
    \label{eq:def-gamma-omega}
\end{align}

\begin{remark}\label{rk:omega-inverse}
If one strengthens \eqref{eq:cost-omega} to the coordinatewise estimate
\[
    |c(x,y)-c(x',y')|\le \omega(\|x-x'\|)+\omega(\|y-y'\|),
\]
then one can replace \eqref{eq:def-varrho-omega} with any radius satisfying $\omega(\varrho_{\eps,\omega})\le \eps/8$. In particular, denoting by $\omega^{-1}$ the (generalized) inverse of $\omega$, one may take the more explicit radius
\[
  \varrho_{\eps,\omega}:=\omega^{-1}(\eps/8)\wedge R.
\]
\end{remark}

\begin{lemma}[Coordinate-wise coercivity with a modulus of continuity]\label{lemma:PL-forVariance-omega}
Let $r_0$ be as in~\eqref{eq:def-C-r0}. For all $r\le r_0$ and $(u,v)\in L^2(P)\times L^2(Q)$,
\[
\int_{f_r\oplus g_r\ge c}(u\oplus v)^2\,d(P\otimes Q)
\ge \alpha_{\eps,\omega}\,\Var_P(u).
\]
\end{lemma}

\begin{proof}
The proof is identical to that of Lemma~\ref{lemma:PL-forVariance} up to \eqref{eq:proof_for_variance_step1}. It remains to replace the estimate following the definition of $m_x$.

Choose a measurable selector
\[
\Omega\ni x\mapsto y(x)\in\argmax_{y\in\Omega'}(f_*(x)+g_*(y)-c(x,y))_+
\]
and define
\[
    m_x:=\max_{y\in\Omega'}(f_*(x)+g_*(y)-c(x,y))_+.
\]
As before, \eqref{eq:FOC} implies $m_x\ge \eps$ for all $x\in\Omega$. Let $x,z\in\Omega$ and $y\in\Omega'$. Using \eqref{eq:cost-omega} and \eqref{eq:potentials-omega}, we obtain
\begin{align*}
    f_*(z)+g_*(y)-c(z,y)
    &\ge f_*(x)+g_*(y(x))-c(x,y(x)) \\
    &\quad -\omega(\|x-z\|)-\omega(\|y(x)-y\|)-\omega(\|x-z\|+\|y(x)-y\|) \\
    &= m_x-\omega(\|x-z\|)-\omega(\|y(x)-y\|)-\omega(\|x-z\|+\|y(x)-y\|).
\end{align*}
Hence, if $z\in B_{\varrho_{\eps,\omega}}(x)$ and $y\in B_{\varrho_{\eps,\omega}}(y(x))$, then by \eqref{eq:def-varrho-omega},
\[
    f_*(z)+g_*(y)-c(z,y)
    \ge \eps-2\omega(\varrho_{\eps,\omega})-\omega(2\varrho_{\eps,\omega})
    \ge \frac{\eps}{2},
\]
so that $z\in\cT_y^\eps$. Thus the inclusion used in the proof of Lemma~\ref{lemma:PL-forVariance} remains valid with $\eps/(8L)$ replaced by $\varrho_{\eps,\omega}$. Repeating the remainder of that proof with $\varrho=\varrho_{\eps,\omega}$ yields the claim.
\end{proof}

\begin{remark}\label{rk:ballContained-omega}
The proof of Remark~\ref{rk:ballContained} is unchanged with $\eps/(8L)$ replaced by $\varrho_{\eps,\omega}$. In particular,
\begin{equation}\label{eq:ballContained-omega}
    \inf_{y\in\Omega'}Q(B_{\varrho_{\eps,\omega}}(y))\le Q(\cS_x^\eps),
    \qquad
    \inf_{x\in\Omega}P(B_{\varrho_{\eps,\omega}}(x))\le P(\cT_y^\eps).
\end{equation}
\end{remark}

\begin{theorem}[Uniform coercivity of $\mathbb M$ under a modulus of continuity]\label{Theorem:Bound-double-omega}
Let $r_0$ be as in~\eqref{eq:def-C-r0}. Then for all $r\le r_0$ and $(u,v)\in L^2(P)\times L^2(Q)$,
\begin{equation}\label{eq:Mr-coercive-omega}
    \langle u\oplus v,\mathbb M(u\oplus v)\rangle_{L^2(P\otimes Q)}
    \ge \beta_{\eps,\omega}\|u\oplus v\|_{L^2(P\otimes Q)}^2,
\end{equation}
and consequently,  for a.e.\ $r\in[0,r_0]$,
\begin{equation}\label{eq:Bound-double-takeaway-omega}
    \|(f_*-f)\oplus(g_*-g)\|_{L^2(P\otimes Q)}^2
    \le \frac{\eps}{\beta_{\eps,\omega}}\phi''(r).
\end{equation}
\end{theorem}

\begin{proof}
The proof of Theorem~\ref{Theorem:Bound-double} carries over once \eqref{eq:Measure-ball} and \eqref{eq:Measure-ball2} are replaced by
\[
    P(\cT_{r,y})\ge \delta_P\min(\varrho_{\eps,\omega},1)^d=\kappa_{\eps,\omega},
    \qquad
    Q(\cS_{r,x})\ge \inf_{y\in\Omega'}Q(B_{\varrho_{\eps,\omega}}(y)).
\]
These follow from Remark~\ref{rk:ballContained-omega} together with Remark~\ref{rk:coneCondition}. Note also that
\[
    \alpha_{\eps,\omega}
    \le \inf_{y\in\Omega'}Q(B_{\varrho_{\eps,\omega}}(y)),
\]
so Step~2 of the proof of Theorem~\ref{Theorem:Bound-double} is unchanged as well. Therefore the same spectral argument yields \eqref{eq:Mr-coercive-omega}, and \eqref{eq:Bound-double-takeaway-omega} then follows from \eqref{eq:phiSecondDeriv} exactly as before.
\end{proof}

\begin{corollary}[PL inequality for $\Phi$ with a modulus of continuity]\label{th:PL-Phi-omega}
Theorem~\ref{th:PL-Phi} remains valid under Assumption~\ref{assumption:cost-omega}, with $\gamma_\eps$ replaced throughout by $\gamma_{\eps,\omega}$.
\end{corollary}

\begin{proof}
Replace Theorem~\ref{Theorem:Bound-double} by Theorem~\ref{Theorem:Bound-double-omega} in the proof of Theorem~\ref{th:PL-Phi}. Since the proof depends on the cost regularity only through the constant $\beta_\eps$, the same argument yields the conclusion with $\beta_{\eps,\omega}$ in place of $\beta_\eps$, hence with $\gamma_{\eps,\omega}=4\beta_{\eps,\omega}^{-1}$ in place of $\gamma_\eps$.
\end{proof}

\begin{corollary}[PL inequality for $\Gamma$ with a modulus of continuity]\label{th:PL-Gamma-omega}
Assume Assumptions~\ref{assumption:marginals} and~\ref{assumption:cost-omega}. Then for every $(f,g)\in L^\infty(\Omega)\times L^\infty(\Omega')$,
\[
    \|f\oplus g-f_*\oplus g_*\|_{L^2(P\otimes Q)}
    \le \gamma_{\eps,\omega}\max\bigl(\|f\oplus g-f_*\oplus g_*\|_\infty,\eps\bigr)
    \|{\rm D}\Gamma(f,g)\|_{L^2(P)\times L^2(Q)},
\]
and
\[
    \|{\rm D}\Gamma(f,g)\|_{L^2(P)\times L^2(Q)}^2
    \ge
    \frac{\QOT_\eps(P,Q)-\Gamma(f,g)}{\gamma_{\eps,\omega}\max\bigl(\|f\oplus g-f_*\oplus g_*\|_\infty,\eps\bigr)}.
\]
\end{corollary}

\begin{proof}
Lemma~\ref{le:translationPhiToGamma} is unchanged. Therefore, Corollary~\ref{th:PL-Phi-omega} implies the stated bounds exactly as in the proof of Theorem~\ref{th:PL-Gamma}.
\end{proof}

\begin{remark}\label{rk:lipschitz-recovered}
If $\omega(r)=Lr$, the explicit radius from \cref{rk:omega-inverse} becomes
\[
    \varrho_{\eps,\omega}=\frac{\eps}{8L}\wedge R
\]
and \eqref{eq:def-gamma-omega} reduces to \eqref{eq:gamma-defn} (up to the truncation at $R$). That is, the results in the present appendix recover the results in the main text when specialized to a Lipschitz cost~$c$. 
\end{remark}

\proofreadhere

\section{Extension to connected Lipschitz supports}
\label{app:PL-connected}

This appendix shows how the convexity assumption on the support $\Omega$ of $P$ can be relaxed to
\begin{equation}\label{eq:Omega-Lipschitz-connected}
\Omega=\overline U,
\qquad
U\subset\R^d \text{ is a bounded connected Lipschitz domain.}
\end{equation}
Apart from this relaxation, we keep the rest of \cref{assumption:marginals}. Since $U$ is Lipschitz, $|\partial U|=0$, and we freely identify Lebesgue integrals over $U$ and over $\Omega$.

\begin{remark}[Lower bounds for ball measures]
\label{rk:ball-measure-quantitative}
There exists a constant $\delta_\Omega\in(0,1]$, depending only on $\Omega$, such that
\begin{equation}
\label{eq:quantitative-thickness}
|B_r(x)\cap \Omega|\ge \delta_\Omega r^d,
\qquad x\in\Omega,\quad 0<r\le \diam(\Omega) .
\end{equation}
Indeed, any bounded  Lipschitz domain satisfies a uniform interior cone condition; see \cite[Theorem~1.2.2]{Grisvard.85}. %
Consequently,
\begin{equation}
\label{eq:ball-measure-quantitative}
P(B_r(x))\ge \widetilde\delta_P\min(r^d,1),
\quad x\in\Omega,\quad r>0, \qquad\text{where }\widetilde\delta_P:=\min(1,\lambda_P\delta_\Omega).
\end{equation}
\end{remark}

The other geometric input needed in the proof is the following nonlocal Poincar\'e estimate, whose proof reduces locally to the convex case. It replaces \cref{Lemma:Bound-energy}.

\begin{lemma}
\label{Lemma:Bound-energy-connected}
There exists a constant $C_\Omega\ge 1$, depending only on $\Omega$, such that for every $\varrho>0$ and every $h\in L^2(\Omega)$,
\begin{equation}
\label{eq:Bound-energy-connected}
\int_\Omega\int_\Omega (h(x)-h(z))^2\,dz\,dx
\le
C_\Omega
\left\lceil \frac{\diam(\Omega)}{\varrho}\right\rceil^{d+2}
\mathcal E_\varrho(h),
\end{equation}
where
\[
\mathcal E_\varrho(h):=
\int_\Omega\int_{B_\varrho(x)\cap\Omega} (h(x)-h(z))^2\,dz\,dx.
\]
\end{lemma}

\begin{proof}
If $\varrho>\diam(\Omega)$, then $B_\varrho(x)\supset\Omega$ for every $x\in\Omega$. If $\varrho=\diam(\Omega)$, this remains true up to a $\cL_{2d}$-null subset of $\Omega\times\Omega$. Hence $\mathcal E_\varrho(h)$ equals the left-hand side of \cref{eq:Bound-energy-connected}, and there is nothing to prove. We assume below that $0<\varrho<\diam(\Omega)$.

Choose a finite collection of open sets $U_1,\dots,U_M$ covering $\Omega$ and set
\[
\Omega_i:=U_i\cap\Omega .
\]
We choose this Lipschitz atlas so that each $\Omega_i$ is, up to a null set, the image of a bounded convex reference set $V_i$ under a bi-Lipschitz map
\[
\Phi_i:V_i\to\Omega_i .
\]
Such an atlas follows from the standard finite Lipschitz atlas for $\overline U$ and compactness of $\overline U$. After a refinement, connectedness of $U$ allows us to choose the atlas so that the overlap graph, with vertices $1,\dots,M$ and edge relation
\[
i\sim j
\quad\Longleftrightarrow\quad
|\Omega_i\cap\Omega_j|>0,
\]
is connected. Let all constants denoted by $C_\Omega$ depend only on this fixed atlas, including the bi-Lipschitz constants, the number of charts, the multiplicity of the cover, and the overlap measures $|\Omega_i\cap\Omega_j|$ along the chosen connected graph. We allow $C_\Omega$ to change from line to line.

Fix $i$ and set $g_i:=h\circ\Phi_i$ on $V_i$. Since $V_i$ is convex, \cref{Lemma:Bound-energy} applied to $V_i$ at scale comparable to $\varrho$ gives
\[
\int_{V_i}\int_{V_i}(g_i(u)-g_i(v))^2\,dv\,du
\le
C_\Omega
\left\lceil \frac{\diam(\Omega)}{\varrho}\right\rceil^{d+2}
\int_{V_i}\int_{V_i\cap B_{\varrho/C_\Omega}(u)}
(g_i(u)-g_i(v))^2\,dv\,du .
\]
Using the identity
\[
\int_{V_i}\int_{V_i}(g_i(u)-g_i(v))^2\,dv\,du
=
2|V_i|\inf_{a\in\R}\int_{V_i}(g_i-a)^2\,du,
\]
we obtain
\[
\inf_{a\in\R}\int_{V_i}(g_i-a)^2\,du
\le
C_\Omega
\left\lceil \frac{\diam(\Omega)}{\varrho}\right\rceil^{d+2}
\int_{V_i}\int_{V_i\cap B_{\varrho/C_\Omega}(u)}
(g_i(u)-g_i(v))^2\,dv\,du .
\]
Changing variables $x=\Phi_i(u)$ and $z=\Phi_i(v)$, and using the bi-Lipschitz bounds, yields
\begin{equation}
\label{eq:local-chart-connected}
\inf_{a\in\R}\int_{\Omega_i}(h-a)^2\,dx
\le
C_\Omega
\left\lceil \frac{\diam(\Omega)}{\varrho}\right\rceil^{d+2}
\int_{\Omega_i}\int_{\Omega_i\cap B_\varrho(x)}
(h(x)-h(z))^2\,dz\,dx .
\end{equation}

Define the local means and local oscillations
\[
m_i:=\frac1{|\Omega_i|}\int_{\Omega_i}h(x)\,dx,
\qquad
E_i:=\int_{\Omega_i}(h-m_i)^2\,dx .
\]
By \cref{eq:local-chart-connected},
\begin{equation}
\label{eq:Ei-bound-connected}
E_i
\le
C_\Omega
\left\lceil \frac{\diam(\Omega)}{\varrho}\right\rceil^{d+2}
\int_{\Omega_i}\int_{\Omega_i\cap B_\varrho(x)}
(h(x)-h(z))^2\,dz\,dx .
\end{equation}
If $i\sim j$, then Jensen's inequality gives
\[
|m_i-m_j|^2
\le
\frac{2}{|\Omega_i\cap\Omega_j|}
\int_{\Omega_i\cap\Omega_j}\bigl(|h-m_i|^2+|h-m_j|^2\bigr)\,dx
\le
C_\Omega(E_i+E_j).
\]
Since the overlap graph is finite and connected,
\[
|m_i-m_1|^2\le C_\Omega\sum_{j=1}^M E_j,
\qquad i=1,\dots,M.
\]
Therefore,
\[
\int_\Omega (h-m_1)^2\,dx
\le
\sum_{i=1}^M\int_{\Omega_i}(h-m_1)^2\,dx
\le
C_\Omega\sum_{i=1}^M E_i .
\]
Combining this with \cref{eq:Ei-bound-connected} and using the finite multiplicity of the atlas,
\[
\inf_{a\in\R}\int_\Omega (h-a)^2\,dx
\le
C_\Omega
\left\lceil \frac{\diam(\Omega)}{\varrho}\right\rceil^{d+2}
\mathcal E_\varrho(h).
\]
Finally,
\[
\int_\Omega\int_\Omega (h(x)-h(z))^2\,dz\,dx
=
2|\Omega|\inf_{a\in\R}\int_\Omega (h-a)^2\,dx,
\]
and the factor $2|\Omega|$ is absorbed into $C_\Omega$.
\end{proof}

We can now repeat the proof of \cref{lemma:PL-forVariance} using \cref{Lemma:Bound-energy-connected} instead of \cref{Lemma:Bound-energy}. This gives the same assertion as in \cref{lemma:PL-forVariance} except that its constant $\alpha$ is replaced by
\begin{equation}
\label{eq:def-alpha-connected}
\widetilde\alpha_\eps
:=
\frac{\lambda_P^2}{\Lambda_P^2}
\frac{q_0}{C_\Omega
\left\lceil \frac{8L\,\diam(\Omega)}{\eps}\right\rceil^{d+2}},
\qquad
q_0:=\inf_{y\in\Omega'}Q\Bigl(B_{\frac{\eps}{8L}}(y)\Bigr)>0.
\end{equation}

The proof of \cref{Theorem:Bound-double} then carries over with modified constants. Specifically, the assertion now holds with $\beta_\eps$ replaced by 
\[
\widetilde\beta_\eps
:=
\frac14 \widetilde\kappa_\eps \widetilde\alpha_\eps,
\qquad
\widetilde\kappa_\eps
:=
\widetilde\delta_P\min\!\left(\frac{\eps}{8L},1\right)^d,
\]
with $\widetilde\delta_P$ from \cref{eq:ball-measure-quantitative} and
$\widetilde\alpha_\eps$ from \cref{eq:def-alpha-connected}.
Finally, we obtain the corresponding PL inequality and error bound.

\begin{theorem}[PL inequality on connected Lipschitz supports]
\label{th:PL-Gamma-connected}
Assume \cref{assumption:marginals}, except that convexity of $\Omega$ is replaced by
\[
\Omega=\overline U
\qquad\text{for a bounded connected Lipschitz domain }U\subset\R^d,
\]
and \cref{assumption:cost}. Then the assertions of \cref{th:PL-Phi,th:PL-Gamma} remain valid with
$\gamma_\eps$ replaced by
\begin{equation}
\label{eq:gamma-connected}
\widetilde\gamma_\eps
=
16\left(\widetilde\delta_P^{-1}\max\!\left(\frac{8L}{\eps},1\right)^d\right)
\frac{\Lambda_P^2}{\lambda_P^2}
\frac{
C_\Omega
\left\lceil \frac{8L\,\diam(\Omega)}{\eps}\right\rceil^{d+2}
}{q_0},
\end{equation}
where $\widetilde\delta_P$ is defined in \cref{eq:ball-measure-quantitative} and
$q_0=\inf_{y\in\Omega'}Q(B_{\eps/(8L)}(y))$.
\end{theorem}

\begin{corollary}[Linear convergence]
If $\gamma_\eps$ is replaced by $\widetilde\gamma_\eps$ of \eqref{eq:gamma-connected}, all assertions of \cref{coro:GD,coro:Implicit,coro:Explicit} carry over to the setting of \cref{th:PL-Gamma-connected}.
\end{corollary}

\bibliographystyle{abbrv}
\bibliography{Biblio}
\end{document}